\documentclass[12pt]{article}

\usepackage{amssymb}
\usepackage{amsmath}
\usepackage{graphicx}
\usepackage{fullpage}
\usepackage{amsthm}
\usepackage{hyperref}
\usepackage{float}

\usepackage{enumerate}

\newtheorem{theorem}{Theorem}
\numberwithin{theorem}{section}
\newtheorem{lemma}[theorem]{Lemma}
\newtheorem{conjecture}[theorem]{Conjecture}
\newtheorem{corollary}[theorem]{Corollary}
\newtheorem{proposition}[theorem]{Proposition}
\newtheorem{algorithm}[theorem]{Insertion Algorithm}
\numberwithin{equation}{section}
\numberwithin{figure}{section}

\makeatletter
\newtheorem*{rep@theorem}{\rep@title}
\newcommand{\newreptheorem}[2]{%
\newenvironment{rep#1}[1]{%
\def\rep@title{#2 \ref{##1}}%
\begin{rep@theorem}}%
{\end{rep@theorem}}}
\makeatother

\newreptheorem{theorem}{Theorem}
\newreptheorem{conjecture}{Conjecture}

\theoremstyle{definition}
\newtheorem{definition}[theorem]{Definition}
\newtheorem{remark}[theorem]{Remark}
\newtheorem{histremark}[theorem]{Historical Remark}

\newcommand{\maj}{\operatorname{maj}}
\newcommand{\area}{\operatorname{area}}
\newcommand{\dinv}{\operatorname{dinv}}
\newcommand{\Des}{\operatorname{Des}}
\newcommand{\ides}{\operatorname{ides}}
\newcommand{\mini}{\operatorname{minimaj}}
\newcommand{\ns}{\hspace{-1pt}}

\title{Schedules and the Delta Conjecture}

\author{James Haglund and Emily Sergel}

\begin{document}

\maketitle

\begin{abstract}
In a recent preprint, Carlsson and Oblomkov (2018) obtain a long sought-after monomial basis 
for the ring $\textnormal{D\ns R}_n$ of diagonal coinvariants.  Their basis is closely related to the ``schedules" formula for the Hilbert series of  $\textnormal{D\ns R}_n$ which was conjectured by 
the first author and Loehr (2005) and first proved by Carlsson and Mellit (2018), as a consequence of their proof of the famous Shuffle Conjecture.  
In this article we obtain a schedules formula for the combinatorial side of the Delta Conjecture, a conjecture introduced by the first author, 
Remmel and Wilson (2018), which contains the Shuffle Theorem as a special case.  Motivated by the Carlsson-Oblomkov basis for $\textnormal{D\ns R}_n$ and our Delta schedules formula,
we introduce a (conjectural) basis for the super-diagonal coinvariant ring $\textnormal{S\ns D\ns R}_n$, an $S_n$-module generalizing $\textnormal{D\ns R}_n$ introduced recently by Zabrocki (2019), which
conjecturally corresponds to the Delta Conjecture.
\end{abstract}


\section{Introduction} \label{sec:intro}

Given a polynomial $f(x_1,\ldots ,x_n,y_1,\ldots ,y_n)  \in \mathbb{C} [x_1,\ldots x_n, y_1, \ldots ,y_n]$, the symmetric group $S_n$ acts diagonally by permuting the $x$- and $y$-variables
identically, i.e., for all $\sigma \in S_n$,
\begin{align}
\sigma f(x_1,\ldots ,x_n,y_1,\ldots ,y_n ) = f(x_{\sigma _1},\ldots ,x_{\sigma _n},y_{\sigma _1},\ldots ,y_{\sigma _n} ).
\end{align}
Let $\textnormal{D\ns R}_n$ denote the {\it diagonal coinvariant ring}, defined as the quotient
\begin{align}
\mathbb{C} [x_1,\ldots x_n, y_1, \ldots ,y_n]/I_n (X,Y),
\end{align}
where $I_n(X,Y)$ is the ideal generated by all $S_n$-invariant  polynomials in \linebreak $\mathbb{C} [x_1,\ldots x_n, y_1, \ldots ,y_n]$ without constant term.

A great deal of research over the last $25$ years in algebra and combinatorics has been devoted to understanding the structure of  $\textnormal{D\ns R}_n$, starting with the original papers of Haiman \cite{nchoose2} and Garsia-Haiman \cite{PFpath} introducing the
topic.   $\textnormal{D\ns R}_n$ is naturally bigraded by homogeneous $x$- and $y$-degrees, so
$\textnormal{D\ns R}_n = \oplus _{i,j \ge 0} \, \textnormal{D\ns R}_n^{(i,j)}$.
A combinatorial description of the Hilbert series 
\begin{align}
\textnormal{Hilb}(\textnormal{D\ns R}_n;q,t) = \sum_{i,j \ge 0} q^i t^j   \textnormal{dim} ( \textnormal{D\ns R}_n^{(i,j)})
\end{align}
in terms of parking functions was 
conjectured by the first author and Loehr \cite{sched}. 
The first author, Haiman, Loehr, Remmel and Ulyanov \cite{HHLRU} extended this to the famous Shuffle Conjecture (now the Shuffle Theorem), which takes into account the $S_n$ action on $\textnormal{D\ns R}_n$ and gives a monomial expansion of the Frobenius characteristic $\textnormal{Frob}(\textnormal{D\ns R}_n;q,t)$.
The Shuffle Theorem was proved only a few years ago by Carlsson and Mellit \cite{CM}.    Here the Frobenius characteristic is defined as
\begin{align}
\textnormal{Frob}(\textnormal{D\ns R}_n;q,t) = \sum_{i,j \ge 0} q^i t^j \sum_{\lambda \vdash n} s_{\lambda}(X) \, 
\textnormal{Mult}(\lambda,\textnormal{D\ns R}_n^{(i,j)}),
\end{align}
where the inner sum is over all partitions $\lambda$ of $n$, $s_{\lambda}$ is the Schur function, and $ \textnormal{Mult}(\lambda,\textnormal{D\ns R}_n^{(i,j)})$ is the multiplicity of the irreducible $S_n$-module corresponding to $\lambda$ in the decomposition of $\textnormal{D\ns R}_n^{(i,j)}$ into irreducible submodules.

In \cite{sched}, the first author and Loehr show that the parking function formula for $\textnormal{Hilb}(\textnormal{D\ns R}_n;q,t)$ can be expressed more compactly as a sum over permutations, where the summand is a power of $t$ times a product
of certain $q$-integers.  This formula has come to be known as the ``schedule" formula for $\textnormal{Hilb}(\textnormal{D\ns R}_n;q,t)$, and was utilized by the second author in her proof \cite{MyThesis,Ser17} 
of a conjecture of Loehr and Warrington involving weighted Dyck paths in an $n \times n$ square.   
Schedules also play an important role in Hicks's Functional Equation Conjecture \cite{HicksThesis}, which has been proved by Garsia, Hicks and Xin \cite{FunctEq}.

In a recent preprint on the math arXiv, Carlsson and Oblomkov \cite{CObasis} obtain a monomial basis for $\textnormal{D\ns R}_n$, which we describe in detail in Section \ref{sec:basis}.    Their basis is closely
related to the schedule formula for $\textnormal{Hilb}(\textnormal{D\ns R}_n;q,t)$. In fact, setting all $x$-variables equal to $q$ and all $y$-variables equal to $t$ reduces their basis to the schedule formula for
$\textnormal{Hilb}(\textnormal{D\ns R}_n;q,t)$, and gives a second proof of that formula (the first proof is as a consequence of the Shuffle Theorem).

In \cite{Delta}, the first author, Remmel and Wilson introduced the Delta Conjecture, which says that
a certain symmetric function involving parameters $q,t$ and $z$ equals a combinatorial sum over parking functions with $q,t,z$-weights.   When $z=0$ it reduces to the Shuffle Theorem. 
 A lot has been written about various special 
cases of the Delta Conjecture (in fact, over $1000$ pages in published articles, arXiv preprints and  Ph.D. theses) but the general conjecture is still open.

Recently Zabrocki \cite{mikeMod} 
introduced an extension of $\textnormal{D\ns R}_n$, which we denote $\textnormal{S\ns D\ns R}_n$ and refer to as the ``super-diagonal coinvariant ring." He conjectures that the Frobenius characteristic of $\textnormal{S\ns D\ns R}_n$ equals a graded sum of the symmetric functions appearing in the Delta Conjecture. (See Conjecture \ref{conj:zabrocki} below.)   This module involves two sets of commuting variables $x_1,\dots,x_n$ and $y_1,\dots,y_n$. The third set of variables $\theta_1,\dots,\theta_n$ is Grassmannian, i.e., 
\begin{align}
\theta _i \theta _j = -\theta _j \theta _i \quad \hbox{ for all }1\le i \le j \le n.
\end{align} 
(Note that this implies $\theta_i^2=0$ for all $i$.) The $\theta$-variables commute with the $x$- and $y$-variables. The symmetric group $S_n$ acts diagonally in all three sets of variables. That is, for all $\sigma \in S_n$,
\begin{align}
\begin{split}
&\sigma f(x_1,\ldots ,x_n,y_1,\ldots ,y_n , \theta _1, \ldots , \theta _n) \\
& \qquad = f(x_{\sigma _1},\ldots ,x_{\sigma _n},y_{\sigma _1},\ldots ,y_{\sigma _n}, \theta _{\sigma _1}, \ldots ,
\theta_{\sigma _n} ).
\end{split}
\end{align}
Let $I_n(X,Y,\Theta)$ be the ideal generated by all $S_n$-invariant polynomials in these three sets of variables without constant term. Then
\begin{align}
\textnormal{S\ns D\ns R}_n = \mathbb{C} [x_1,\ldots x_n, y_1, \ldots ,y_n, \theta _1, \ldots , \theta _n]/I_n(X,Y,\Theta).
\end{align}

Associating the $\theta$-degree with the parameter $z$, Zabrocki's Conjecture says that the Frobenius characteristic 
$\textnormal{Frob}(\textnormal{S\ns D\ns R}_n;q,t,z)$ equals a graded sum of either side of the Delta Conjecture.  It is not yet known whether any two of these three ``faces" of the Delta Conjecture are equal, but 
Zabrocki's Conjecture now embeds the Delta Conjecture firmly in the study of coinvariant algebras, significantly increasing its relevance.

Motivated by the new monomial basis for $\textnormal{D\ns R}_n$ as well as the module interpretation for the Delta Conjecture, in this article we obtain a schedule formula
for the Delta Conjecture and use it to construct a candidate basis for $\textnormal{S\ns D\ns R}_n$.

The rest of the paper is organized as follows. In Section \ref{sec:PF}, we provide the necessary combinatorial background to state the Shuffle Theorem and Delta Conjecture. In Section \ref{sec:sched}, we start by giving combinatorial insight to the first author and Loehr's  \cite{sched} schedule formula (Theorem \ref{thm:PFsched}) and then present a Delta-version.

\begin{reptheorem}{thm:MPFsched}
For every ordered set partition $\Pi$, 
\begin{equation*}
\sum_{M\ns P\ns F \in \mathcal{M\ns P\ns F}(\Pi)} t^{\area(M\ns P\ns F)} q^{\dinv(M\ns P\ns F)} = t^{\maj(\tau(\Pi))} \prod_{c} \left[ w_\Pi(c) \right]_q,
\end{equation*}
where the product is over all $c$ occurring in any block of $\Pi$.
\end{reptheorem}
See also Definitions \ref{def:tauPi} and \ref{def:markedschednum}. This schedule formula suggests the following conjecture.

\begin{repconjecture}{conj:SDRhilb}
For all $n \geq 1$,
\begin{equation*}
\textnormal{Hilb}(\textnormal{S\ns D\ns R}_n;q,t) = \sum_{\Pi} \, z^{n-|\Pi|} \, t^{\maj(\tau(\Pi))} \prod_{c=1}^n \left[ w_\Pi(c) \right]_q
\end{equation*}
where the sum ranges over all ordered set partitions of $[n]$ and $|\Pi|$ is the number of blocks in $\Pi$.
\end{repconjecture}

In Section \ref{sec:basis} we introduce a candidate basis for $\textnormal{S\ns D\ns R}_n$ based on Conjecture \ref{conj:SDRhilb} and inspired by Carlsson and Oblomkov's \cite{CObasis} basis for $\textnormal{D\ns R}_n$ (which is based on Theorem \ref{thm:PFsched}). See Conjecture \ref{conj:basis} and the text preceding it. The subsections of Section \ref{sec:basis} give different types of evidence for the validity of our candidate basis. In particular, the combinatorics of marked parking functions leads us to the following conjecture.

\begin{repconjecture}{conj:ineq}
Let $a$, $b$, $c$ be non-negative integers. Then the homogeneous component of $\textnormal{S\ns D\ns R}_n$ of order $n$ with $x$-degree $a$, $y$-degree $b$, and $z$-degree $c$ is non-empty if and only if
\begin{equation*}
a + b + {c+1 \choose 2} \leq {n \choose 2}
\end{equation*}
\end{repconjecture}

We prove one direction of this conjecture for the alternants of $\textnormal{S\ns D\ns R}_n$; see Theorem \ref{thm:bigIneq}.


\section{The Delta Conjecture} \label{sec:PF}

In this section we present the important combinatorial objects and results leading up to the valley-version of the Delta Conjecture \cite{Delta}. The Delta Conjecture is a generalization of the Shuffle Theorem of Carlsson and Mellit \cite{CM}. The Shuffle Theorem was originally conjectured by the first author, Haiman, Loehr, Remmel and Ulyanov \cite{HHLRU}. It expresses $\textnormal{Frob}(\textnormal{D\ns R}_n)$ as a weighted sum of parking functions.

\subsection{Parking Functions}

Parking functions were originally introduced by Konheim and Weiss \cite{KW} as special functions. Here we follow Garsia and Haiman \cite{PFpath} and identify them with labeled lattice paths.

A Dyck path of size $n$ is a sequence of North and East steps going from $(0,0)$ to $(n,n)$ staying weakly above the line $y=x$. A parking function $P\ns F$ of size $n$ consists of a Dyck path of size $n$ and labels (called ``cars") which are adjacent to the path's North steps and increase from bottom to top in each column. Typically the set of labels is $[n]$. For example, Figure \ref{fig:Ex0-PF} shows a parking function of size 8. However it will be convenient here to allow a parking function of size $n$ to have any $n$ distinct, positive labels.

\begin{figure}
\begin{center}
\includegraphics[width=1.0in]{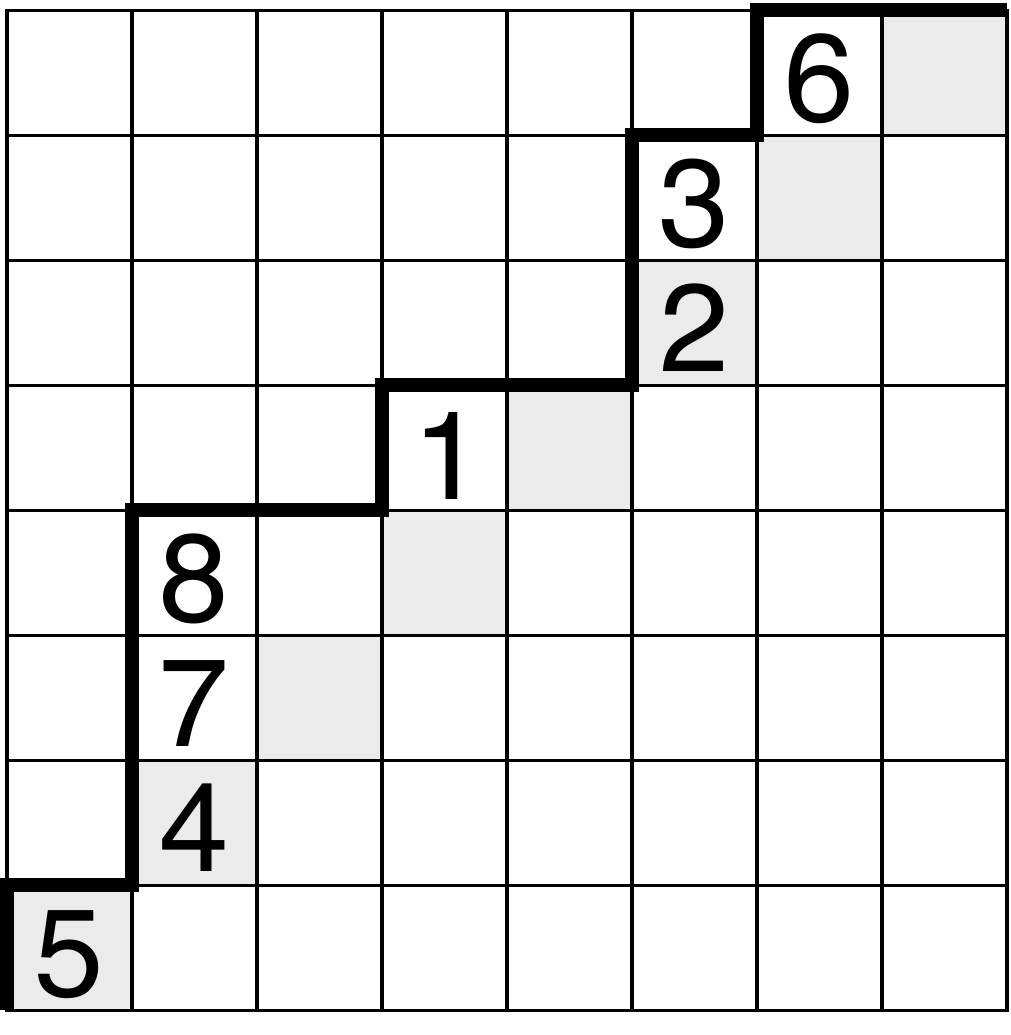}
\caption{A parking function of size 8.}
\label{fig:Ex0-PF}
\end{center}
\end{figure}

To each parking function $P\ns F$, we associate four statistics: two non-negative integers $\area(P\ns F)$ and $\dinv(P\ns F)$, a permutation $\sigma(P\ns F)$ of $P\ns F$'s labels, and a set $\ides(P\ns F)$. (Throughout we use ``permutation" to mean a word consisting of any distinct, positive numbers.) The area of $P\ns F$ is simply the number of full squares between $P\ns F$'s Dyck path and the line $y=x$. For the parking function in Figure \ref{fig:Ex0-PF}, $\area=6$.

The remaining statistics all use the notion of diagonals. The cells cut by the line $y=x$ (which are shaded in the figure) together compose the main diagonal or $0$-diagonal. The cells immediately above that form the $1$-diagonal. Immediately above that is the $2$-diagonal, etc. For instance in Figure \ref{fig:Ex0-PF}, the $0$-diagonal contains cars $2$, $4$, and $5$, the $1$-diagonal contains cars $1$, $3$, $6$, and $7$, the $2$-diagonal contains car $8$, and all higher diagonals are empty. Note that a car in the $i$-diagonal contributes $i$ to the area of the parking function.

The word of $P\ns F$, denoted $\sigma(P\ns F)$, is the permutation obtained by reading $P\ns F$'s cars from highest to lowest diagonal, and from right to left within each diagonal. In Figure \ref{fig:Ex0-PF}, the word is $8 \, 6 \, 3 \, 1 \, 7 \, 2 \, 4 \, 5$. From this we create the inverse descent set $\ides(P\ns F)$. Namely, $\ides(P\ns F) = \Des(\sigma(P\ns F)^{-1}) = \{ i : i+1 \hbox{ occurs left of } i \hbox{ in } \sigma(P\ns F) \}$. In the figure, $\ides = \{ 2, 5, 7\}$.

Finally, $\dinv$ counts the number of ``diagonal inversions" in a parking function. A pair of cars $(s,b)$ create a diagonal inversion if $s<b$ and either (1) cars $s$ and $b$ are in the same diagonal with $b$ further right or (2) car $b$ is in the diagonal above car $s$ and $b$ is further to the left. In Figure \ref{fig:Ex0-PF}, the pair $(1,3)$ creates an inversion of the first kind, called a primary diagonal inversion, and the pair $(1,8)$ creates an inversion of the second kind, called a secondary diagonal inversion. Altogether, in this example there are 3 primary diagonal inversions (namely $(1,3)$, $(1,6)$, and $(3,6)$) and 4 secondary diagonal inversions (namely $(2,7)$, $(1,8)$, $(3,8)$, and $(6,8)$). Hence $\dinv = 7$.

We see in the next subsection how these are used to enumerate parking functions in a way that is meaningful in representation theory, symmetric function theory, and the study of Macdonald polynomials.

\subsection{The Shuffle Theorem}

Let $\mathcal{P\ns F}_n$ be the set of parking functions of size $n$ with labels $[n]$. The Shuffle Theorem (below) uses $\area$, $\dinv$, and $\ides$ to expand a certain symmetric function $\nabla e_n$ in terms of Gessel's fundamental basis for the ring of quasi-symmetric functions \cite{Gessel}. 

\begin{theorem}[The Shuffle Theorem \cite{CM}] \label{thm:shuffle}
For $n \geq 1$,
\begin{equation}
\nabla e_n = \sum_{P\ns F \in \mathcal{P\ns F}_n} t^{\area(P\ns F)} q^{\dinv(P\ns F)} F_{\ides(P\ns F)}.
\end{equation}
Here $\nabla$ is a linear operator defined on the modified Macdonald polynomial basis as 
\begin{align}
\nabla {\tilde H}_{\mu}(X;q,t) =  t^{n(\mu)} q^{n(\mu^{\prime})}{\tilde H}_{\mu}(X;q,t),
\end{align}
with $n(\mu) = \sum_i (i-1)\mu _i$ and $\mu'$ the conjugate of the partition $\mu$.  
\end{theorem}
The expressions in Theorem \ref{thm:shuffle} are connected to the representation theory of $\textnormal{D\ns R}_n$ through the following famous result of Haiman \cite{DHfrob}.
\begin{theorem}[\cite{DHfrob}] \label{thm:DHfrob}
For $n \geq 1$, 
\begin{equation}
\textnormal{Frob}(\textnormal{D\ns R}_n;q,t) = \nabla e_n.
\end{equation}
\end{theorem}

Hence the Shuffle Theorem can be seen as having three ``sides": the combinatorial side is an enumeration of parking functions, the symmetric function side is $\nabla e_n$, and the representation-theoretic side is provided by Theorem \ref{thm:DHfrob}.

\subsection{The Delta Conjecture}

Shortly before the proof of the Shuffle Theorem, the first author, Remmel and Wilson \cite{Delta} introduced a substantial generalization of the first two sides of the Shuffle Theorem known as the Delta Conjecture. (The third side was only recently added by Zabrocki \cite{mikeMod}.) The symmetric function side of this conjecture relies on a family of operators known as Delta operators. In their full generality, the definition of these operators is based on plethystic calculus. For our purposes, we need only the following special case.

For any nonnegative integer $k$ and a polynomial $P$ with $m$ terms, $P=\sum_{i=1}^m t_i$, let
\begin{equation}
e_k[P] = \sum_{1 \leq i_1 < \cdots < i_k \leq m} t_{i_1} \cdots t_{i_k}.
\end{equation}
In particular, $e_0[P]=1$, $e_1[P]=P$, $e_{m}[P]$ is a single monomial, and $e_{> m}[P]=0$. Let $\mu$ be any partition and define
\begin{align}
B_{\mu}(q,t) = \sum_{c \in \mu} q^{a^{\prime}(c)} t^{\ell^{\prime}(c)}.
\end{align}
Here the sum is over all cells $c$ in the Ferrers diagram of $\mu$, and $a^{\prime}(c)$ and $\ell^{\prime}(c)$ are the co-arm and co-leg of $c$, respectively. If $\mu$ is a partition of $n$, then $B_\mu(q,t)$ is a polynomial in $q$ and $t$ with $n$ terms. Moreover if $\mu$ is not the empty partition, then it must include the corner box and hence one term of $B_\mu(q,t)$ is $1$. For all $k$, we define the following two linear operators on symmetric functions according to their action on the modified Macdonald basis.
\begin{align}
\begin{split}
&\Delta_{e_k} \tilde{H}_{\mu}[X;q,t] = e_k[B_{\mu}(q,t)] \, \tilde{H}_{\mu}[X;q,t],\\
&\Delta^{\prime}_{e_k}  \tilde{H}_{\mu}[X;q,t] =
e_k[B_{\mu}(q,t)-1] \, \tilde{H}_{\mu}[X;q,t].
\end{split}
\end{align}

\begin{histremark}
The $\nabla$ operator was first introduced by Bergeron and Garsia \cite{SciFi} \emph{after} Haiman's Theorem \ref{thm:DHfrob} was already conjectured in a different form. This led them, along with Haiman and Tesler, to further introduce the Delta operators \cite{Tesler}. The modified Macdonald polynomials, which are central to all of these formulations, were introduced by Garsia and Haiman \cite{modmac} in order to investigate the Macdonald Positivity Conjecture \cite{macoriginal} from a representation-theoretic perspective.
\end{histremark}

Notice that $\Delta^{\prime}_{e_0}$ is the identity operator and $\Delta^{\prime}_{e_m} f = 0$ if the degree of $f$ is less than or equal to $m$. It is also not hard to see that on the space of homogeneous symmetric functions of degree $n$, $\Delta^{\prime}_{e_{n-1}} = \Delta_{e_n} = \nabla$. The symmetric function side of the Delta Conjecture is simply $\Delta^{\prime}_{e_{n-k-1}} \, e_n$, which indeed specializes to the Shuffle Theorem's symmetric function side when $k=0$.

There are actually two versions of the combinatorial side of the Delta Conjecture which use two variations of parking functions. Here we use only the ``valley version''. In a parking function, a valley occurs whenever an East step is followed by a North step and either (1) there is no car under the East step or (2) the car under the East step is smaller than the car adjacent to the North step. In Figure \ref{fig:Ex0-PF}, there are three valleys. A valley-marked parking function is a parking function with markings on some of its valleys. We also say that the car adjacent to a marked valley's North step is marked. We can create $2^3$ valley-marked parking functions by adding markings to all subsets of Figure \ref{fig:Ex0-PF}'s valleys. One such valley-marked parking function is shown in Figure \ref{fig:Ex0-VMPF}. In \cite{Delta}, there are also rise-marked parking functions, but we will not use them here. Hence we will often refer to a valley-marked parking function simply as a marked parking function.

\begin{figure}
\begin{center}
\includegraphics[width=1.0in]{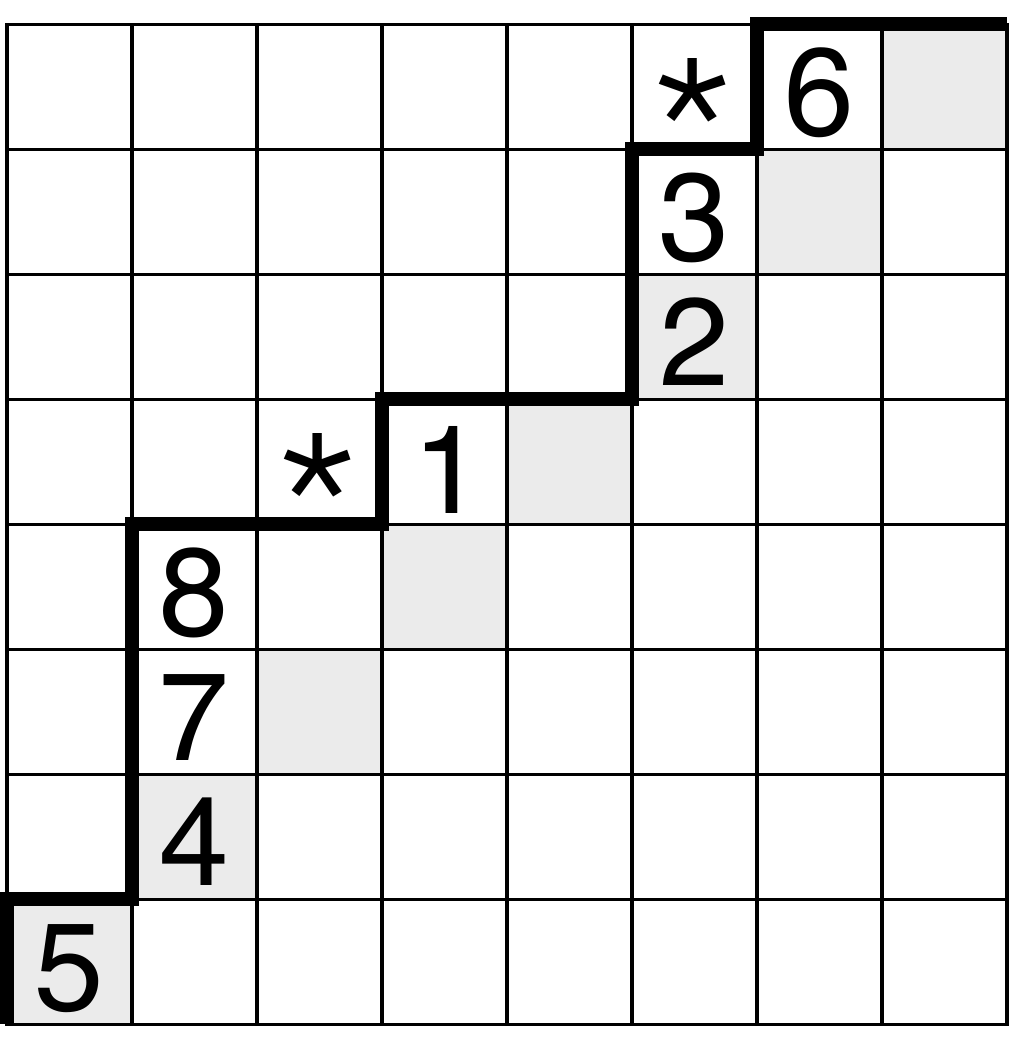}
\caption{A valley-marked parking function of size 8 with 2 marked valleys.}
\label{fig:Ex0-VMPF}
\end{center}
\end{figure}

The $\area$, word, and $\ides$ statistics are identical for a marked parking function and its underlying (unmarked) parking function. Only the $\dinv$ statistic changes. 

\begin{definition} \label{def:markeddinv}
Let $M\ns P\ns F$ be a valley-marked parking function. Two cars $(s,b)$ in $M\ns P\ns F$ create a primary diagonal inversion if $s$ and $b$ lie in the same diagonal, $s$ is further left, and $s$ is unmarked. The cars $(s,b)$ create a secondary diagonal inversion if $b$ is in the diagonal above $s$, $b$ is further left, and $b$ is unmarked. We define $\dinv($M\ns P\ns F$)$ to be the number of primary and secondary diagonal inversions minus the number of marked valleys.
\end{definition}

Notice that the positions that create a diagonal inversion are the same as for unmarked parking functions, but that we require the car further to the left to be unmarked. In Figure \ref{fig:Ex0-PF} we have $7$ diagonal inversions: $(1,3)$, $(1,6)$, $(3,6)$, $(2,7)$, $(1,8)$, $(3,8)$, $(6,8)$. In Figure \ref{fig:Ex0-VMPF}, the pairs $(1,3)$ and $(1,6)$ no longer create diagonal inversions because the car further left is marked. Hence there are 5 pairs creating diagonal inversions and 2 marked cars, giving $\dinv = 5-2=3$.

\begin{remark} \label{rmk:dinv}
Even though the calculation of $\dinv$ involves subtraction, it always yields a nonnegative number. For any marked car $c$, consider the closest unmarked car to the left of $c$ in the same diagonal. Either it is smaller than $c$ (creating a primary diagonal inversion), or it is larger and must have some even larger, unmarked car on top of it (creating a secondary diagonal inversion).
\end{remark}

Let $\mathcal{M\ns P\ns F}_{n,k}$ be the set of marked parking functions of size $n$ with $k$ marked valleys and labels $[n]$.
\begin{conjecture}[The Delta Conjecture, valley version \cite{Delta}]
For all $k \leq n$,
\begin{equation}
\Delta^{\prime}_{e_{n-k-1}} e_n = \sum_{M\ns P\ns F \in \mathcal{M\ns P\ns F}_{n,k}} t^{\area(M\ns P\ns F)} q^{\dinv(M\ns P\ns F)} F_{\ides(M\ns P\ns F)}.
\end{equation}
\end{conjecture}

As previously mentioned, Zabrocki \cite{mikeMod} has recently added a representation theoretical side to the Delta Conjecture.

\begin{conjecture}[\cite{mikeMod}] \label{conj:zabrocki}
For all $n \geq 1$,
$$
\textnormal{Frob}(\textnormal{S\ns D\ns R}_n;q,t,z) = \sum_{k=0}^{n} z^k \Delta^{\prime}_{e_{n-k-1}} e_n.
$$
\end{conjecture}

Combining these two conjectures gives a formula for $\textnormal{Frob}(\textnormal{S\ns D\ns R}_n;q,t,z)$ in terms of valley-marked parking functions enumerated by $\area$, $\dinv$, $\ides$, and the number of markings.


\section{Schedules} \label{sec:sched}

It is well-known that for any $S_n$-module $M$,
\begin{align}
\langle \textnormal{Frob}(M;q,t), h_1 ^n \rangle = \textnormal{Hilb}(M;q,t).
\end{align}
This follows from the fact that the coefficient of $m_{1^n}$ in the Schur function $s_{\lambda}$ equals the dimension of the irreducible $S_n$-module corresponding to $\lambda$.  Hence by Theorems 
\ref{thm:shuffle} and \ref{thm:DHfrob} we have the following.

\begin{corollary}
For $n \geq 1$,
\begin{equation}
\textnormal{Hilb}(\textnormal{D\ns R}_n;q,t) = \langle \nabla e_n , h_1^n\rangle = \sum_{P\ns F \in \mathcal{P\ns F}_n} t^{\area(P\ns F)} q^{\dinv(P\ns F)}.
\end{equation}
\end{corollary}

The rightmost equality in this identity was originally conjectured by the first author and Loehr \cite{sched}. 
In \cite{sched} there is also a notion of schedules, which groups parking functions together so that the enumeration $\sum t^{\area(P\ns F)} q^{\dinv(P\ns F)}$ within each group is a simple product. This gives a compact formula for the Hilbert series of $\textnormal{D\ns R}_n$.

Each of these groups is the result of building up parking functions one car at a time through an insertion process that behaves nicely with respect to the statistics $\area$ and $\dinv$. In Subsection \ref{subsec:buildPF}, we present these ideas to provide insight for the following subsection. In Subsection \ref{subsec:buildMPF}, we generalize to marked parking functions, and obtain nice conjectural formulas for $\langle \Delta^{\prime}_{e_k} e_n, h_1^n \rangle$ and for the Hilbert series of $\textnormal{S\ns D\ns R}_n$.

\begin{histremark}
The term ``schedules" was introduced by Adriano Garsia and Angela Hicks while working on Hicks's thesis \cite{HicksThesis}. The idea behind this name is that the schedule tells you the order (or schedule) of cars being inserted to build up a parking function.
\end{histremark}

\subsection{Schedules for Parking Functions} \label{subsec:buildPF}

\begin{algorithm}
\label{alg:PFinsert}
Let $P\ns F$ be a parking function, let $c$ be a car not already present in $P\ns F$, and let $k$ be such that the $(k+1)$-diagonal of $P\ns F$ is empty and the $k$-diagonal of $P\ns F$ contains no car smaller than $c$. We create several new parking functions by inserting $c$ into $P\ns F$ in each of the following ways:
\begin{itemize}
\item Let $s<c$ be a car in the $(k-1)$-diagonal of $P\ns F$. Move all cars which are in a higher row than $s$ up and to the right once. Place car $c$ directly above car $s$.
\item Let $b>c$ be a car in the $k$-diagonal of $P\ns F$. Move all cars which are in a higher row than $b$ up and to the right once. Place car $c$ directly above and to the right of car $b$.
\item If $k=0$, move all cars up and to the right once. Place car $c$ in the lower-lefthand corner.
\end{itemize}
Let the set of all these new parking functions be denoted $\operatorname{Insert}(P\ns F,c,k)$.
\end{algorithm}

For example, consider the parking function $P\ns F$ on the left side of Figure \ref{fig:PFinsert}. On the right side of Figure \ref{fig:PFinsert} we show the elements of $\operatorname{Insert}(P\ns F,5,2)$, and circle those cars creating new diagonal inversions with the $5$.

\begin{figure}
\begin{center}
\includegraphics[width=4.6in]{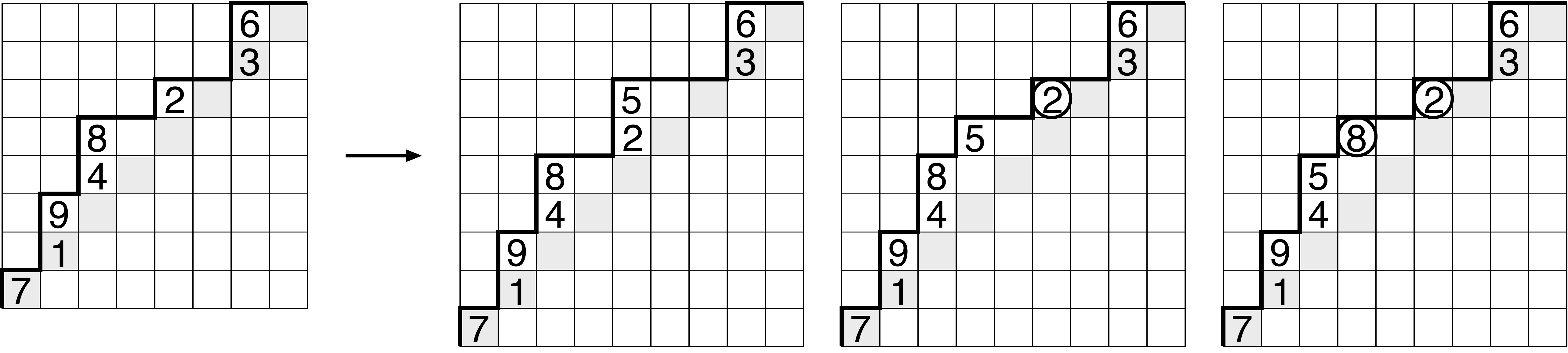}
\caption{A parking function together with those obtained by inserting a $5$ into the $2$-diagonal. The cars creating new diagonal inversions with $5$ are circled.}
\label{fig:PFinsert}
\end{center}
\end{figure}

For any parking function $P\ns F$ and any car $c$ in the highest non-empty diagonal of $P\ns F$, let $\rho_c(P\ns F)$ be the parking function obtained by deleting $c$ from $P\ns F$ and shifting all cars in higher rows down and to the left once. The following proposition is self-evident.

\begin{proposition} \label{prop:alginverse}
Let $P\ns F$ be any nonempty parking function. There is exactly one choice of $P\ns F'$, $c$, and $k$ for which $P\ns F \in \operatorname{Insert}(P\ns F',c,k)$. In particular, $k$ is the number of the highest non-empty diagonal in $P\ns F$, $c$ is the smallest car in the $k$-diagonal of $P\ns F$, and $P\ns F' = \rho_c(P\ns F)$. 
\end{proposition}

So we will use Insertion Algorithm \ref{alg:PFinsert} to build up parking functions, and each parking function will be built in exactly one way. Moreover, this method of construction has the following nice properties with respect to $\area$ and $\dinv$.

\begin{lemma} \label{lemma:qan}
For any parking function $P\ns F$, car $c$, and diagonal $k$ for which the Insertion Algorithm \ref{alg:PFinsert} applies,
$$
\sum_{P\ns F' \in \operatorname{Insert}(P\ns F,c,k)} t^{\area(P\ns F')} q^{\dinv(P\ns F')}
= t^{\area(P\ns F)+k} q^{\dinv(P\ns F)} \big[ \left| \operatorname{Insert}(P\ns F,c,k) \right| \big]_q.
$$
Here $[n]_q = 1 + q + q^2 + \dots + q^{n-1}$.
\end{lemma}

\begin{proof}
It is clear that each $P\ns F' \in \operatorname{Insert}(P\ns F,c,k)$ has $\area(P\ns F') = \area(P\ns F)+k$. All cars other than $c$ stay in their original diagonals and keep the same relative positions from left to right. Hence the diagonal inversions of $P\ns F$ are preserved, and any new diagonal inversions must involve the new car $c$.

Notice that the cars $s$ or $b$ by which we inserted $c$ during the Insertion Algorithm \ref{alg:PFinsert} are precisely those cars which might create new diagonal inversions with $c$. That is, the cars $s<c$ in the previous diagonal may create secondary diagonal inversions, while the cars $b>c$ in the same diagonal may create primary diagonal inversions. In both cases, the diagonal inversion will be created exactly when $c$ is the car further to the left. So when we insert $c$ at the rightmost position, it creates no diagonal inversions with any cars. Each time we move $c$ further left among these possible positions, it creates one new diagonal inversion with car at the insertion spot that has just been passed over. 

Hence the $\dinv$ of these $P\ns F'$ are $\dinv(P\ns F)$, $\dinv(P\ns F)+1$, $\dinv(P\ns F)+2$, \dots, $\dinv(P\ns F)+\left| \operatorname{Insert}(P\ns F,c,k) \right|-1$ as claimed.
\end{proof}

Let $\tau$ be any permutation of length $n$. We build a tree of parking functions level by level, starting with the parking function of size $1$ with car $\tau_n$ at the root at stage 1. At the $i$-th stage, let $c=\tau_{n-i}$ and $k$ be such that $c$ lies in the $(k+1)$-st from last run of $\tau$. Then to each parking function $P\ns F$ which is currently a leaf of our tree, we add a child for each element of $\operatorname{Insert}(P\ns F,c,k)$. For example, if $\tau=2314$, we obtain the tree in Figure \ref{fig:buildPF}.

\begin{figure}
\begin{center}
\includegraphics[width=3in]{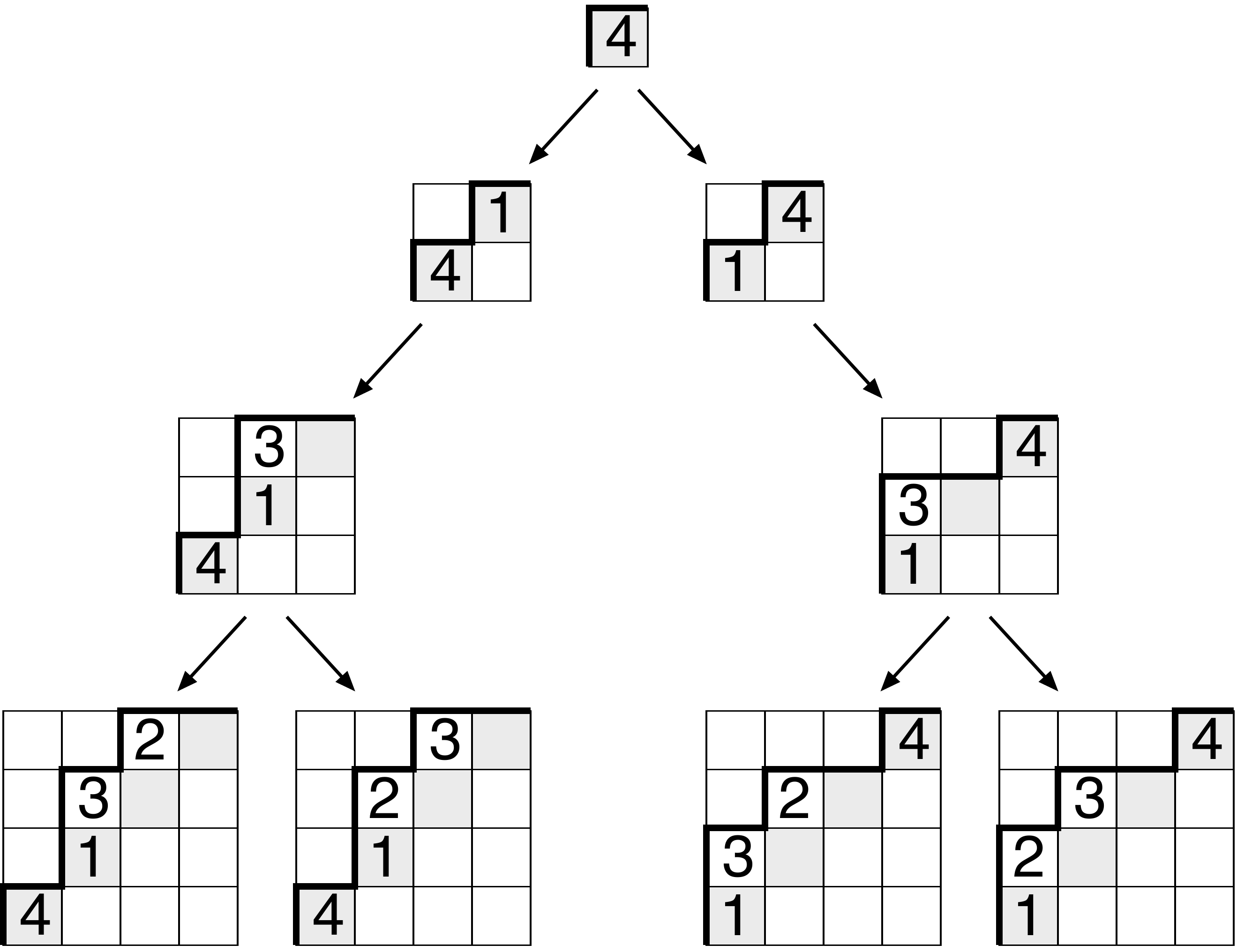}
\caption{The tree of parking functions built from the schedule $\tau=2314$.}
\label{fig:buildPF}
\end{center}
\end{figure}

It is not hard to see that the leaves of the tree generated by $\tau$ contains all parking functions whose $0$-diagonal contains the cars in the last run of $\tau$ (in some order), whose $1$-diagonal contains the cars in the second to last run of $\tau$, and so on. With this in mind, we define the schedule of a parking function $P\ns F$ to be the permutation $\tau(P\ns F)$ formed by writing the cars from each diagonal of $P\ns F$ in increasing order, from highest to lowest diagonal. That is, the leaves of the tree generated by the permutation $\tau$ are all the parking functions with schedule $\tau$. Because of the connection here between permutations and parking functions, we will often refer to elements of permutations as cars.

Iterating Lemma \ref{lemma:qan} tells us that the sum $\sum_{P\ns F} t^{\area(P\ns F)} q^{\dinv(P\ns F)}$ taken over all the parking functions with schedule $\tau$ equals a power of $t$ times a product of $q$-analogs. It is not hard to see that the $t$-power is simply $\maj(\tau)$. Moreover, these $q$-analogs correspond to the degrees of the parking functions at each level of the tree corresponding to $\tau$. For example, the weights of the four leaves of Figure \ref{fig:buildPF} sum to $t^2 \, [2]_q^2$. We call these degrees schedule numbers, and we can read them off of the schedule in the following way.

Let $\tau$ be any permutation. We append a $0$ to the end of $\tau$. Then for each (nonzero) car $c$ in $\tau$, the schedule number $w_\tau(c)$ equals the number of cars larger than $c$ in the same run plus the number of cars smaller than $c$ in the following run. These are exactly the cars used to insert $c$ in Insertion Algorithm \ref{alg:PFinsert}, hence they are the degrees of the tree built from $\tau$. The appended $0$ accounts for the extra point of insertion in the special case $k=0$.

All these observations lead us to the following theorem.
\begin{theorem}[\cite{sched}] \label{thm:PFsched}
For every permutation $\tau$, 
\begin{equation}
\sum_{P\ns F \in \mathcal{P\ns F}(\tau)} t^{\area(P\ns F)} q^{\dinv(P\ns F)} = t^{\maj(\tau)} \prod_{c} \left[ w_\tau(c) \right]_q,
\end{equation}
where the product is over all $c$ occurring in $\tau$.
\end{theorem}

\subsection{Schedules for Valley-Marked Parking Functions} \label{subsec:buildMPF}

The goal of this subsection is to obtain a Delta-analog of Theorem \ref{thm:PFsched}. We will repeat the structure of the previous subsection, starting with an insertion algorithm for marked cars.

The success of Insertion Algorithm \ref{alg:PFinsert} hinged on two facts: First, the potential insertion sites corresponded exactly to the cars that could create diagonal inversions with the new car. Second, those new diagonal inversion pairs consistently had the new car on the left. This allowed us predict the change to $\dinv$ as we pushed the new car along the potential insertion sites.

We need an algorithm for inserting marked cars that has similar properties. Recall from Definition \ref{def:markeddinv} that a marked car $c$ can only create a diagonal inversion with either (1) a bigger unmarked car $b$ lying in the diagonal above $c$ and to its left or (2) a smaller unmarked car $s$ lying in the same diagonal as $c$ and to its left. The fact that the bigger cars $b$ lie in a higher diagonal and the fact that the marked car $c$ is consistently on the right together suggest that we will first insert all the unmarked cars into all of the diagonals (using Insertion Algorithm \ref{alg:PFinsert}) and only then start inserting marked cars. Because of the presence of cars in higher diagonals, we will have to be more careful during the insertion itself.

\begin{algorithm}
\label{alg:markedinsert}
Let $M\ns P\ns F$ be a marked parking function, let $c$ be a car not already present in $M\ns P\ns F$, and let $k$ be such that the $k$-diagonal of $P\ns F$ contains no marked car smaller than $c$. For each unmarked car $s<c$ in the $k$-diagonal of $M\ns P\ns F$ or $b>c$ in the $(k+1)$-diagonal, we create a new marked parking function by inserting a marked $c$ into $M\ns P\ns F$ in the following way: Look to the right of this unmarked car $s$ or $b$ for the next time that either an unmarked car $b'>c$ appears in the $(k+1)$-diagonal or the underlying Dyck path returns to the line $y=x+k$. (For a car $s<c$ with nothing directly above it, the latter will happen immediately.)
\begin{itemize}
\item If you first encounter an unmarked car $b'>c$ in the $(k+1)$-diagonal, shift $b'$ and every car in a higher row up and to the right once. Place the marked car $c$ directly under car $b'$.
\item If you first encounter a place where the underlying Dyck path returns to the line $y=x+k$, move all cars above this point up and to the right once. Place the marked car $c$ in the cell directly to the right and above this point, so that it lies in the $k$-diagonal.
\end{itemize}
Let the set of all these new parking functions be denoted $\operatorname{Insert}^{*}(M\ns P\ns F,c,k)$.
\end{algorithm}

For example, Figure \ref{fig:insert} shows a marked parking function $M\ns P\ns F$ on the left. On the right we see the elements of $\operatorname{Insert}^{*}(M\ns P\ns F,5,1)$. The cars creating new diagonal inversions with $5$ are circled.

\begin{figure}
\begin{center}
\includegraphics[width=4.6in]{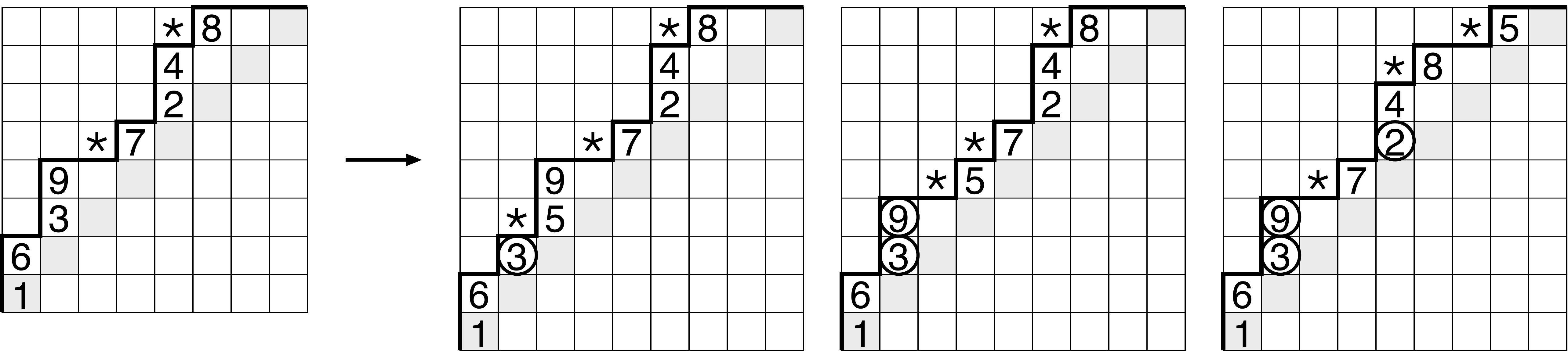}
\caption{A marked parking function together with those obtained by inserting a marked 5 into the $1$-diagonal. The cars creating new diagonal inversions with $5$ are circled.}
\label{fig:insert}
\end{center}
\end{figure}

\begin{lemma} \label{lemma:markedqan}
For any marked parking function $P\ns F$, car $c$, and diagonal $k$ for which the Insertion Algorithm \ref{alg:markedinsert} applies,
\begin{align*}
&\sum_{M\ns P\ns F' \in \operatorname{Insert}^{*}(M\ns P\ns F,c,k)} \hspace{-0.2in} t^{\area(M\ns P\ns F')} q^{\dinv(M\ns P\ns F')} \\
& \qquad \qquad \quad = t^{\area(M\ns P\ns F)+k} q^{\dinv(M\ns P\ns F)} \big[ \left| \operatorname{Insert}^{*}(M\ns P\ns F,c,k) \right| \big]_q.
\end{align*}
\end{lemma}

\begin{proof}
The proof is very similar to that of Lemma \ref{lemma:qan}. The main difference is that $c$ always makes a diagonal inversion with the car $b$ or $s$ near which it is inserted, as well as all potential insertion sites $b$ and $s$ lying to its left. However $c$ is inserted so that it does not create a diagonal inversion with anything to the right of $s$ or $b$. We see in Definition \ref{def:markeddinv} that we must also subtract one from $\dinv$ to account for the new marked car $c$. Hence as we move $c$ from \emph{left to right}, we get marked parking functions with $\dinv$ equal to $\dinv(M\ns P\ns F)+1-1$, $\dinv(M\ns P\ns F)+2-1$, \dots, $\dinv(M\ns P\ns F) + \left| \operatorname{Insert}^{*}(M\ns P\ns F,c,k) \right|-1$ as desired.
\end{proof}

We extend $\rho_c$ to marked parking functions in the natural way: If $c$ is any marked car of a marked parking function $M\ns P\ns F$, then $\rho_c(M\ns P\ns F)$ is formed by deleting $c$ (and its mark) and shifting all higher cars down once and to the right. Note that if $c$ lies under some car $b$ and there is another car $s$ in $c$'s diagonal and directly to its left, then in $\rho_c(M\ns P\ns F)$, $b$ will now lie on top of $s$. But since $M\ns P\ns F$ is column-increasing, $c<b$, and since $c$ is marked, $s<c$. So $\rho_c(M\ns P\ns F)$ is also column-increasing.

We have to be slightly careful when generalizing Proposition \ref{prop:alginverse}. 
There may be more than one choice of $c$ and $k$ so that a given $M\ns P\ns F$ is in $\operatorname{Insert}^{*}(\rho_c(M\ns P\ns F),c,k)$. This is because we can insert marked cars from different diagonals in any order. We choose the same order that we used for unmarked cars: we will insert marked cars from lowest to highest diagonal, and from largest to smallest within each diagonal.

\begin{proposition} \label{prop:markedalginverse}
Let $M\ns P\ns F$ be any nonempty marked parking function. There is exactly one choice of $M\ns P\ns F'$, $c$, and $k$ for which $M\ns P\ns F \in \operatorname{Insert}^{*}(P\ns F',c,k)$ \emph{and} for which no diagonal above the $k$-diagonal contains any marked cars. In particular, $k$ is the number of the highest diagonal in $M\ns P\ns F$ containing any marked cars, $c$ is the largest marked car in the $k$-diagonal of $M\ns P\ns F$, and $M\ns P\ns F' = \rho_c(M\ns P\ns F)$. 
\end{proposition}

Let $\tau^{*}$ be a permutation with some elements marked. We build a tree of marked parking functions by inserting the unmarked cars of $\tau^{*}$ from right to left using Insertion Algorithm \ref{alg:PFinsert} and then the marked cars from right to left using Insertion Algorithm \ref{alg:markedinsert}. In either case we use the runs of $\tau^{*}$ to give the diagonal number: if the car $c$ lies in the $(k+1)$-st from last run of $\tau^{*}$, it will be inserted into the $k$-diagonal.

For example, consider $\tau^{*} = 8\overset{*}{1}3\overset{*}{6}7245$. We build a tree whose root is the parking function of size one containing the car $5$. Then we insert the unmarked cars $4$ and $2$ into the $0$-diagonal, then $7$ and $3$ into the $1$-diagonal, $8$ into the $2$-diagonal, and finally the marked cars $6$ and $1$ in the $1$-diagonal. Using the results of Subsection \ref{subsec:buildPF}, we see that there will already be $2^3 \cdot 3^2 = 72$ leaves in the insertion tree when we are done with the unmarked cars. Then the marked $6$ will be inserted in one of two ways (corresponding to cars $8$ and $3$) and finally the marked $1$ has only one potential insertion site (corresponding to car $8$). We cannot show the full insertion tree. Instead, Figure \ref{fig:build} shows the ancestors of the marked parking function in Figure \ref{fig:Ex0-VMPF} with their children. For each of the nodes pictured, we write the statistics $q^{\dinv} t^{\area}$ to the right.

\begin{figure}
\begin{center}
\includegraphics[width=2.6in]{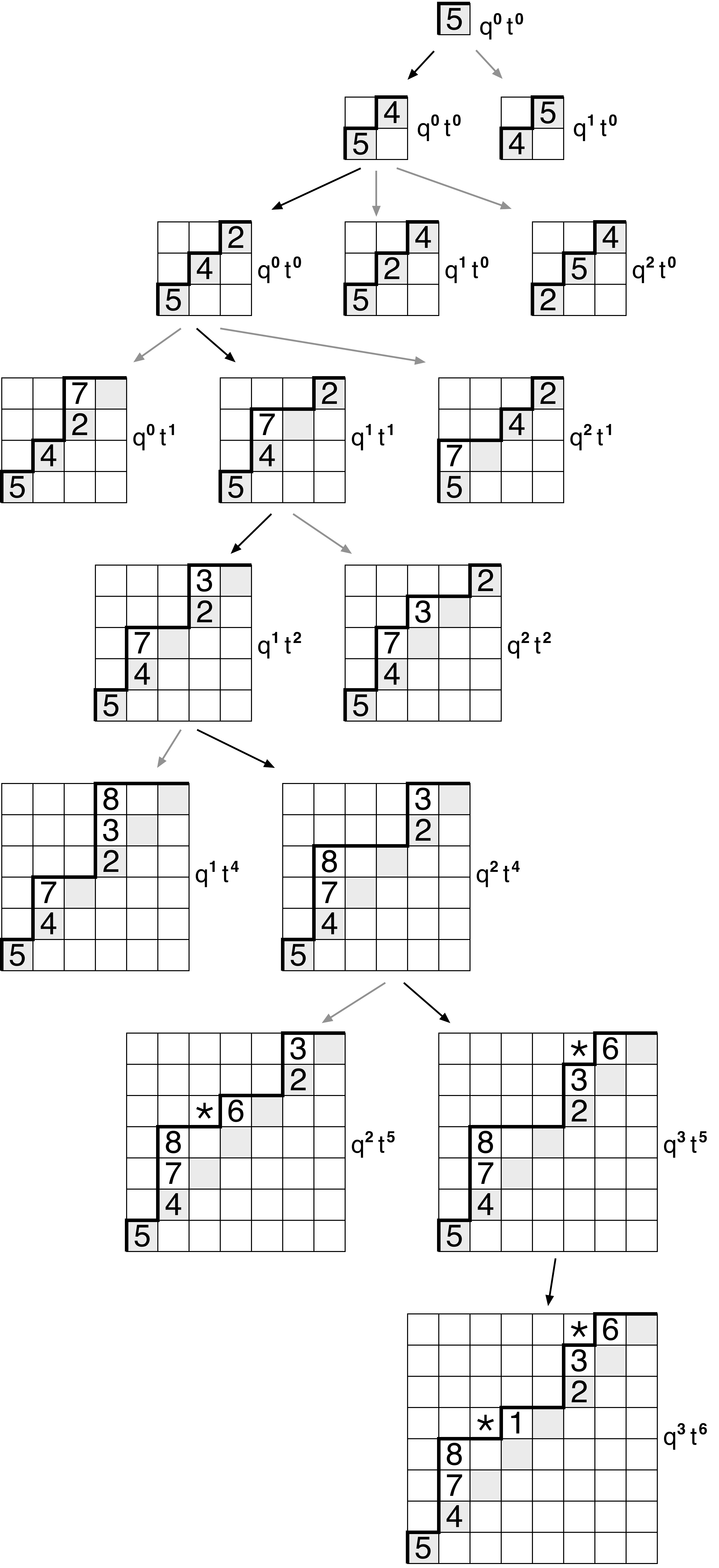}
\caption{Building Figure \ref{fig:Ex0-VMPF} using Insertion Algorithms \ref{alg:PFinsert} and \ref{alg:markedinsert}.}
\label{fig:build}
\end{center}
\end{figure}

To any marked parking function $M\ns P\ns F$, we can associate a marked permutation $\tau^{*}(M\ns P\ns F)$ so that $M\ns P\ns F$ is in its insertion tree. We simply record the sets of cars (marked and unmarked) in each diagonal from smallest to largest, and from highest to lowest diagonals. It is clear from Proposition \ref{prop:markedalginverse} that $M\ns P\ns F$ will not appear any other insertion tree.

However, not every marked permutation of length $n$ has marked parking functions of size $n$ in its insertion tree. For example, consider $\tau^{*} = \overset{*}{2} 1$. There is no way to insert the marked car $2$ in the $1$-diagonal. Let us say for the moment that a marked permutation of size $n$ is \emph{insertable} if we can insert every car and arrive at some non-empty set of marked parking functions of size $n$. For instance, $1 \overset{*}{2}$ and $2 1 \overset{*}{3}$ are insertable while $\overset{*}{2} 1$ and $1 \overset{*}{3} 2$ are not.

It turns out that insertable marked permutations correspond to ordered set partitions, which appear in connection with the Delta Conjecture in \cite{Delta}. For brevity, we will usually write an ordered set partition as a word with bars separating the elements in consecutive blocks. For instance, $235|1679|48$ represents the ordered set partition $\left( \{2,3,5\}, \{1,6,7,9\}, \{4,8\} \right)$.

\begin{definition} \label{def:tauPi}
Let $\Pi$ be an ordered set partition of some finite set $S \subseteq \mathbb{N}$. We form a word $\tau(\Pi)$ from $\Pi$ by an inductive procedure: If $\Pi$ has only one block, $\tau(\Pi)$ is the word obtained by writing the elements of this block in increasing order. Otherwise, let $\Pi'$ be the ordered set partition obtained by removing $\Pi$'s first block $B$. If the first element of $\tau(\Pi')$ is $r$, then $\tau(\Pi)$ is obtained by writing all of those elements of $B$ which are bigger than $r$ in increasing order, followed by all elements of $B$ which are less than $r$ in increasing order, followed by $\tau(\Pi')$. We also define $\tau^{*}(\Pi)$ by adding markings to every number which is not the left-most element of its block in $\tau(\Pi)$.
\end{definition}

For example, we build up $\tau(235|1679|48)$ one block at a time:
\begin{center}
\begin{tabular}{c|c}
$\Pi$ & $\tau(\Pi)$\\
\hline
$48$  & $48$ \\
$1679|48$ & $679148$ \\
$235|1679|48$ & $235679148$
\end{tabular} 
\end{center}
From this we get that $\tau^{*}(235|1679|48) =  2 \overset{*}{3} \overset{*}{5} 6 \overset{*}{7} \overset{*}{9} \overset{*}{1} 4 \overset{*}{8}$.

Note that the number of blocks in $\Pi$ is the number of unmarked elements of $\tau^{*}(\Pi)$. Furthermore, we can reconstruct $\Pi$ from $\tau^{*}(\Pi)$ by introducing block breaks before each unmarked element. Hence there is at most one ordered set partition $\Pi$ yielding a given marked word $\tau^{*}(\Pi)$. Also note that for any ordered set partition $\Pi$, $\maj(\tau(\Pi))=\mini(\Pi)$, as defined in \cite{Delta}. The minimaj statistic was also studied in conjunction with the Delta Conjecture by Rhoades \cite{rhoades} and by Benkart et al. \cite{minimajcrystal}.

\begin{theorem} \label{thm:insertablePi}
A marked permutation is insertable if and only if it equals $\tau^{*}(\Pi)$ for some ordered set partition $\Pi$.
\end{theorem}

We postpone the proof of this theorem to the next subsection. While the proof is not hard, it is a distraction from our goal of generalizing Theorem \ref{thm:PFsched}.

We will refer to ordered set partitions as schedules. We say that the \emph{type} of a marked parking function $M\ns P\ns F$ is the marked permutation whose $i$th run consists of those cars lying in $M\ns P\ns F$'s $i$th highest diagonal and which has a marking on $c$ exactly when car $c$ is marked in $M\ns P\ns F$. For example, the type of the marked parking function in Figure \ref{fig:Ex0-VMPF} is $8 \overset{*}{1} 3 \overset{*}{6} 7 2 4 5$. For a given ordered set partition $\Pi$ of $[n]$, let $\mathcal{M\ns P\ns F}(\Pi)$ be the set of all parking functions of type $\tau^{*}(\Pi)$. Hence the marked parking function in Figure \ref{fig:Ex0-VMPF} is an element of $\mathcal{M\ns P\ns F}(18|36|7|2|4|5)$. By Proposition \ref{prop:markedalginverse}, every marked parking function belongs to exactly one insertion tree. Hence by Theorem \ref{thm:insertablePi}, every marked parking function belongs to $\mathcal{M\ns P\ns F}(\Pi)$ for a unique ordered set partition $\Pi$. Hence
\begin{equation}
\mathcal{M\ns P\ns F}_{n,k} = \bigsqcup_\Pi \mathcal{M\ns P\ns F}(\Pi)
\end{equation}
where $\Pi$ runs over all ordered set partitions of $[n]$ with $n-k$ blocks.

By applying Lemmas \ref{lemma:qan} and \ref{lemma:markedqan} repeatedly, we see for any ordered set partition $\Pi$, the sum
$$
\sum_{M\ns P\ns F} t^{\area(M\ns P\ns F)} q^{\dinv(M\ns P\ns F)}
$$
taken over all the marked parking functions of type $\tau^{*}(\Pi)$ equals a power of $t$ times a product of $q$-analogs. It is not hard to see that the $t$-power is simply $\maj(\tau^{*}(\Pi))$. Moreover, these $q$-analogs correspond to the degrees of the parking functions at each level of the tree corresponding to $\tau^{*}(\Pi)$. For example, we construct part of the tree of type $8 \overset{*}{1} 3 \overset{*}{6} 7 2 4 5$ in Figure \ref{fig:build}. The full tree would have 144 leaves with weights summing to $t^6 [2]_q^4 [3]_q^2$.

\begin{definition} \label{def:markedschednum}
Let $\Pi$ be any ordered set partition. For each $c$ appearing in $\Pi$, we defined the schedule number $w_\Pi(c)$ as follows. We first form the marked permutation $\tau^{*}(\Pi)$ and append a $0$ to the end of it. If $c$ is unmarked, then $w_\Pi(c)$ equals the number of unmarked cars larger than $c$ in the same run of $\tau^{*}(\Pi)$ plus the number of unmarked cars smaller than $c$ in the next run. If $c$ is marked, then $w_\Pi(c)$ equals the number of unmarked cars smaller than $c$ in the same run of $\tau^{*}(\Pi)$ plus the number of unmarked cars larger than $c$ in the previous run. These are exactly the cars used to insert $c$ in Insertion Algorithms \ref{alg:PFinsert} and \ref{alg:markedinsert}, hence they are the degrees of the tree built from $\tau^{*}(\Pi)$. As in Subsection \ref{subsec:buildPF}, the appended $0$ accounts for the extra point of insertion in the special case $k=0$ of Insertion Algorithm \ref{alg:PFinsert} and does not affect the schedule numbers of marked cars.
\end{definition}

For example, if $\Pi=235|1679|48$, then $\tau^{*}(\Pi) = 2 \overset{*}{3} \overset{*}{5} 6 \overset{*}{7} \overset{*}{9} \overset{*}{1} 4 \overset{*}{8}$ has runs $2 \overset{*}{3} \overset{*}{5} 6 \overset{*}{7} \overset{*}{9}$ and $\overset{*}{1} 4 \overset{*}{8}$. From this, we calculate $w_\Pi(c)=2$ if $c=1,7,9$ and $w_\Pi(c)=1$ otherwise.

All these observations lead us to the following theorem.
\begin{theorem} \label{thm:MPFsched}
For every ordered set partition $\Pi$, 
\begin{equation}
\sum_{M\ns P\ns F \in \mathcal{M\ns P\ns F}(\Pi)} t^{\area(M\ns P\ns F)} q^{\dinv(M\ns P\ns F)} = t^{\maj(\tau(\Pi))} \prod_{c} \left[ w_\Pi(c) \right]_q,
\end{equation}
where the product is over all $c$ occurring in any block of $\Pi$.
\end{theorem}

Aside from the extra complication about insertability (which was postponed to Subsection \ref{subsec:insertability}) everything here corresponds perfectly with the results of \cite{sched} for schedules and unmarked parking functions. Theorem \ref{thm:MPFsched} is the desired analog of Theorem \ref{thm:PFsched}. (In fact, Theorem \ref{thm:MPFsched} specializes to Theorem \ref{thm:PFsched} when all of $\Pi$'s blocks are singletons.) Conjecturally, it give a compact combinatorial expression for the Hilbert series of the super-diagonal coinvariant ring $\textnormal{S\ns D\ns R}_n$.

\begin{conjecture} \label{conj:SDRhilb}
For all $n \geq 1$,
\begin{equation}
\textnormal{Hilb}(\textnormal{S\ns D\ns R}_n;q,t) = \sum_{\Pi} \, z^{n-|\Pi|} \, t^{\maj(\tau(\Pi))} \prod_{c=1}^n \left[ w_\Pi(c) \right]_q
\end{equation}
where the sum ranges over all ordered set partitions of $[n]$ and $|\Pi|$ is the number of blocks in $\Pi$.
\end{conjecture}

\subsection{Insertability and ordered set partitions} \label{subsec:insertability}

We return to the proof of Theorem \ref{thm:insertablePi}, starting with the following lemma.

\begin{lemma} \label{lemma:insertable}
A marked permutation $\tau^{*}$ is insertable if and only if does not begin with a marked car and the shorter, unmarked permutation $\widehat{\tau}$ obtained by deleting the marked cars of $\tau^{*}$ has the same number of runs.
\end{lemma}

\begin{proof}
First suppose that $\tau^{*}$ starts with a marked car. If we are able to insert all the other cars of $\tau^{*}$, then we must use Insertion Algorithm \ref{alg:markedinsert} to insert a marked car without any smaller unmarked car in the same diagonal and without any car at all in the diagonal above. This is impossible.

Now suppose that $\widehat{\tau}$ has fewer runs than $\tau$. Either $\tau^{*}$ contains an entire run of marked cars, or there is some run of $\widehat{\tau}$ which contains cars that belong to two different runs of $\tau^{*}$.

Suppose $\tau^{*}$ contains an entire run of marked cars. Consider the earliest such run. Since we have already discussed the case when the first car of $\tau^{*}$ is marked, we may assume this is not the first run of $\tau^{*}$. Hence the run before this exists and contains an unmarked car. So as we are inserting unmarked cars according to Insertion Algorithm \ref{alg:PFinsert}, we will insert them into the $0$-diagonal, the $1$-diagonal, and so on, up through the $k$-diagonal for some $k$. Then, without inserting anything into the $(k+1)$-diagonal, we skip ahead and try to insert an unmarked car into the $(k+2)$-diagonal. This is also impossible.

Suppose there is some run of $\widehat{\tau}$ which contains cars that belong to two different runs of $\tau^{*}$. Let $s < b$ be two such cars which are adjacent in $\widehat{\tau}$ but lie in different runs of $\tau^{*}$. Since we have already handled the case when a whole run is made of marked cars, we may assume that the runs of $s$ and $b$ are adjacent. Suppose $b$ lies in the $(k+1)$-st from last run of $\tau^{*}$. Then after inserting $b$, we have some parking functions whose highest nonempty diagonal is the $k$-diagonal, and the smallest car in this diagonal is $b$. Since there is no car smaller than $s$ in the $k$-diagonal and no car at all in the $(k+1)$-diagonal, this is again impossible.

On the other hand, suppose $\tau^{*}$ does not start with a marked car and that $\widehat{\tau}$ has the same number of runs as $\tau^{*}$ does. The portion of $\tau^{*}$'s insertion tree coming from unmarked cars equals the insertion tree of $\widehat{\tau}$, which can be constructed without any issues arising. It remains to insert the marked cars of $\tau^{*}$ from right to left.

Suppose as we proceed through the marked cars of $\tau^{*}$, we must insert a marked car $c$ into the $k$-diagonal of a marked parking function $M\ns P\ns F$ using Insertion Algorithm \ref{alg:markedinsert}. We would like to see when this produces the empty set. Suppose the $(k+1)$-diagonal of $M\ns P\ns F$ is nonempty, and consider the leftmost car in this diagonal. If it is larger than $c$, then $c$ has at least one insertion site. If it is smaller than $c$, then the car beneath it is also smaller than $c$, and again $c$ has at least one insertion site. On the other hand, if the $(k+1)$-diagonal is empty, that means that $c$ lies in the first run of $\tau^{*}$. Since the first car of $\tau^{*}$ is unmarked, that means there is some smaller unmarked car in the $k$-diagonal of $M\ns P\ns F$. Hence each marked $c$ can be inserted. So $\tau^{*}$ is insertable, as desired.
\end{proof}

\begin{proof}[Proof of Theorem \ref{thm:insertablePi}]
Let $\Pi$ be an ordered set partition. Because the leftmost car of each block of $\Pi$ is unmarked, the first element of $\tau^{*}(\Pi)$ is unmarked. Every contiguous substring of marked cars in $\tau^{*}(\Pi)$ comes from a single block of $\Pi$. Furthermore each of these marked cars either belongs to the same run as the unmarked car before it or the unmarked car after it. (For those coming from the rightmost block, these marked cars always belong to the same run as the unmarked car preceding them.) So there is no run consisting completely of marked cars.

Suppose there is any descent with a marked car next to it. Then that single unmarked car from the same block of that marked car must be larger than the unmarked car from the next block to the right - otherwise this block would have been sorted differently. So when we remove the unmarked cars of $\tau^{*}(\Pi)$, this descent persists. Hence by Lemma \ref{lemma:insertable}, $\tau^{*}(\Pi)$ is indeed insertable.

Let $\tau^{*}$ be some insertable marked permutation. Add a bar before each unmarked element of $\tau^{*}$ (except the first one) and remove markings to obtain an ordered set partition $\Pi$. We would like to see that $\tau^{*} = \tau^{*}(\Pi)$. Suppose not. There is some rightmost block for which $\tau^{*}(\Pi)$ disagrees with $\tau^{*}$. A moment's thought shows that any ordering besides the one used in $\tau^{*}(\Pi)$ introduces unnecessary descents into $\tau^{*}$, thus making it un-insertable by Lemma \ref{lemma:insertable}.
\end{proof}


\section{A Candidate Basis for 
$\textnormal{S\ns D\ns R}_n$
} \label{sec:basis}

We first describe the Carlsson-Oblomkov basis for $\textnormal{D\ns R}_n$.  Let  $\tau \in S_n$ and set
\begin{align}
Y_{\tau} = \prod _{i : \tau _i >\tau_{i+1} \atop 1\le i \le n-1} y_{\tau_1}y_{\tau _2}\cdots y_{\tau _i}.
\end{align}
For $\tau = 25713486$, $Y_{\tau} = y_2 y_5 y_ 7 \, y_2 y_5 y_7 y_1 y_3 y_4 y_8$.  Note that $\tau$ can be reconstructed from $Y_{\tau}$.
Furthermore, let 
\begin{align}
X_{\tau} = \prod_{i=1}^{n-1} (1+x_{\tau _i} + x_{\tau _i}^2 +\ldots + x_{\tau _i}^{w_\tau(i) -1})
\end{align}
where the $w_\tau(i)$ are the schedule numbers from Subsection \ref{subsec:buildPF}.

\begin{theorem}[{\cite[Theorem~3]{CObasis}}]
\label{thm:CObasis}
A monomial basis for $\textnormal{D\ns R}_n$ can be obtained by taking the union over $\tau \in S_n$ 
of all the $\prod_{i=1}^{n-1} w_{i}(\tau)$ monomials occurring in $Y_{\tau} \, X_{\tau}$.   
\end{theorem}

\begin{remark} If we set all $x_i=0$ in the Carlsson-Oblomkov basis, the only terms that survive are the union of the $Y_{\tau}$, which is known as the Garsia-Stanton basis \cite{GaSt84} for the coinvariant algebra $\textnormal{R}_n$ in one set of variables.  If we set all the $y_i=0$, the only $\tau$ that contributes is 
the identity $\tau = 12\cdots n$.  For this $\tau$, $w_i(\tau ) = i$, and the monomials in $X_{\tau}$ form the original basis for $\textnormal{R}_n$ obtained by Artin
\cite{Artin}.
\end{remark}

Following Carlsson and Oblomkov, we now use our schedule formula to construct a candidate basis for the module $\textnormal{S\ns D\ns R}_n$ defined in the introduction.
To each ordered set partition $\Pi$ we will associate $\prod w_\Pi(i)$-many monomials, where the $w_\Pi(i)$ are the schedule numbers from Definition \ref{def:markedschednum}.
To make the formulas below more readable, we write $\tau=\tau(\Pi)$ and $\tau^{*}=\tau^{*}(\Pi)$. 

Let $Y_{\Pi}=Y_{\tau}$ as above. Note that the total degree of $Y_{\Pi}$ is $\maj(\tau)$, which is the $\area$ of every marked parking function in $\mathcal{M\ns P\ns F}(\Pi)$. Let $\Theta_\Pi$ be the product of all $\theta_{\tau^{*}_i}$ for which $\tau^{*}_i$ is marked. (Note that $\Theta_\Pi$ is actually an ordered product. We will take the product of marked elements in $\tau^{*}$ from left to right, but the choice of order does not matter.) If $\Pi$ is a partition of $[n]$ into $k$ blocks, the degree of $\Theta_\Pi$ is $n-k$. Finally, we multiply $Y_\Pi$ and $\Theta_\Pi$ by $X_\Pi = \prod_i [w_\Pi(i)]_{x_i} = \prod_i \left(1+x_i + x_i^2 + \cdots + x_i^{w_\Pi(i)-1}\right)$. Let $\mathcal{B}_\Pi$ be the set of monomials in the expansion of this product. Note that each element of $\mathcal{B}_\Pi$ corresponds to some marked parking function with $\area=y$-degree, $\dinv=x$-degree, and number of markings $=\theta$-degree.

We can reconstruct $\tau$ from $Y_\Pi$, and then use $\Theta_\Pi$ to further reconstruct $\tau^{*}$ and therefore $\Pi$. So these sets are disjoint. Let $\mathcal{B}_n$ be the union of $\mathcal{B}_\Pi$ for all ordered set partitions $\Pi$ of $[n]$.

\begin{conjecture} \label{conj:basis}
For $n \geq 1$, $\mathcal{B}_n$ forms a basis for $\textnormal{S\ns D\ns R}_n$.
\end{conjecture}

\begin{remark}
If we set all $y_i=0$, $\text{S\ns D\ns R}_n$ reduces to $\text{SR}_n$, the module of super-coinvariants, containing one commuting and one anti-commuting set of variables.
Currently no basis for this module is known, and furthermore it is still an open conjecture of Zabrocki
(a precursor of Conjecture \ref{conj:zabrocki})
that the Frobenius characteristic of this module is given by the $t=0$ case of the Delta Conjecture.  
In \cite{HRS}, the first author, Rhoades and Shimozono introduce certain quotient rings $R_{n,k}$ of $\mathbf Q [x_1,\ldots ,x_n]$ whose Frobenius characteristic
does equal the combinatorial side of the $t=0$ case of the Delta Conjecture, up to applying $\omega$ and a simple operation called ``$q$-reversal".  The first author, Rhoades and Shimozono also
obtain a few specific bases for $R_{n,k}$, which also involve ordered set partitions.  Even using these results though, we do not currently know how to prove that when
all $y_i=0$, our candidate basis for $\text{S\ns D\ns R}_n$ reduces to a basis for $\text{SR}_n$.
\end{remark}

We conclude by exploring some properties of $\mathcal{B}_n$, and presenting some evidence for Conjecture \ref{conj:basis} and Conjecture \ref{conj:zabrocki}.

\subsection{Empty and non-empty homogeneous components}

\begin{conjecture} \label{conj:ineq}
Let $a$, $b$, $c$ be non-negative integers. Then the homogeneous component of $\textnormal{S\ns D\ns R}_n$ of order $n$ with $x$-degree $a$, $y$-degree $b$, and $z$-degree $c$ is non-empty if and only if
\begin{equation}
\label{eq:degreeb}
a + b + {c+1 \choose 2} \leq {n \choose 2}
\end{equation}
\end{conjecture}

In \cite{nchoose2}, Haiman gives an elementary proof of this fact for $c=0$. This proof does not seem to generalize for all $c$. Wallach recently gave a proof of the $b=0$ case of Conjecture \ref{conj:ineq} \cite[Theorem 14]{wallach}. (He credits the statement of this theorem to Marino Romero.) In the next subsection, we show that $\textnormal{S\ns D\ns R}_n$'s alternants only occur in degrees satisfying the inequality in Corollary \ref{cor:degrees}. We also see in this section that all elements of our combinatorial basis obey this inequality, and that the inequality is tight in a strong sense.  

\begin{remark}
In a recent preprint \cite{RhWi19} Brendon Rhoades and Andy Wilson introduce an $S_n$-module, defined via operators
acting on a generalization of the 
Vandemonde determinant, which they conjecture is isomorphic to $\text{S\ns D\ns R}_n$.  Furthermore, in a private communication with the 
authors, Rhoades has noted that their module provably satisfies the degree bound \eqref{eq:degreeb}, giving further evidence for the truth of Conjecture \ref{conj:ineq}.
\end{remark}

It is well-known that for parking functions $P\ns F$ of size $n$, $\area(P\ns F) + \dinv(P\ns F) \leq {n \choose 2}$. There is a simple injection from pairs of diagonal inversions to squares above the underlying Dyck path which proves this. We modify this injection to show our stronger inequality for marked parking functions.

\begin{theorem} \label{thm:ineq}
If $M\ns P\ns F$ is a marked parking function of size $n$ with $k$ markings, then
\begin{equation}
\area(M\ns P\ns F) + \dinv(M\ns P\ns F) + {k+1 \choose 2} \leq {n \choose 2}.
\end{equation}
\end{theorem}

\begin{proof}

Let $M\ns P\ns F$ be a marked parking function of size $n$ with $k$ markings.

Suppose that $c_1$ and $c_2$ are two cars creating a diagonal inversion and that $c_1$ is in a column strictly to the left of $c_2$. Since $c_1$ and $c_2$ create a diagonal inversion, $c_1$ is in the same diagonal as $c_2$ or one diagonal higher. The North step adjacent to $c_1$ goes above $c_1$'s diagonal, and the path eventually returns to the lower diagonal of $c_2$. Hence there must be a place between $c_1$ and $c_2$'s North steps where an East step comes back to $c_1$'s diagonal (i.e., it is directly above a cell in that diagonal). Consider the leftmost such East step. Let $\phi(c_1,c_2)$ be the cell which is in the same column as this East step and in the same row as $c_2$. Note that $\phi$ is injective: We can reconstruct $c_1$ by going down from $\phi(c_1,c_2)$ to an East step and then traveling left along the diagonal. To reconstruct $c_2$, simply go right from $\phi(c_1,c_2)$. Hence the image of $\phi(c_1,c_2)$ is a set of cells above the underlying Dyck path which is in bijection with pairs of diagonal inversions, i.e., there are $\dinv(M\ns P\ns F) + k$ elements of the image.

Suppose instead that $c_1$ and $c_2$ are any two marked cars with $c_1$ in a column strictly to the left of $c_2$'s column. We will define a cell $\psi(c_1,c_2)$ which is above the underlying Dyck path and not in the image of $\phi$. Recall that two marked cars cannot create a diagonal inversion. Hence if $c_1$'s diagonal is weakly higher than $c_2$'s, we can use the same algorithm as for $\phi$ to construct $\psi(c_1,c_2)$. On the other hand, if $c_1$ is in a lower diagonal than $c_2$, we instead backtrack to the East step immediately preceding $c_1$ and then let $\psi(c_1,c_2)$ be the cell in this column and in the same row as $c_2$. Again we can reconstruct $c_1$ and $c_2$ from $\psi(c_1,c_2)$: Follow the column down to an East step and right to a North step and the car $c_2$. If the cell below the East step is in a diagonal weakly above that of $c_2$, travel left along that diagonal until you reach $c_1$ as before. Otherwise this East step is part of a marked valley and $c_1$ is immediately to the right. Furthermore this cell is not in the image of $\phi$ since no cell left of $c_1$ and in the same diagonal can make a diagonal inversion with $c_2$. Hence $\psi$'s image consists of ${k \choose 2}$ cells above the underlying Dyck path and it is disjoint from $\phi$'s image.

\begin{figure}
\begin{center}
\includegraphics[width=1.2in]{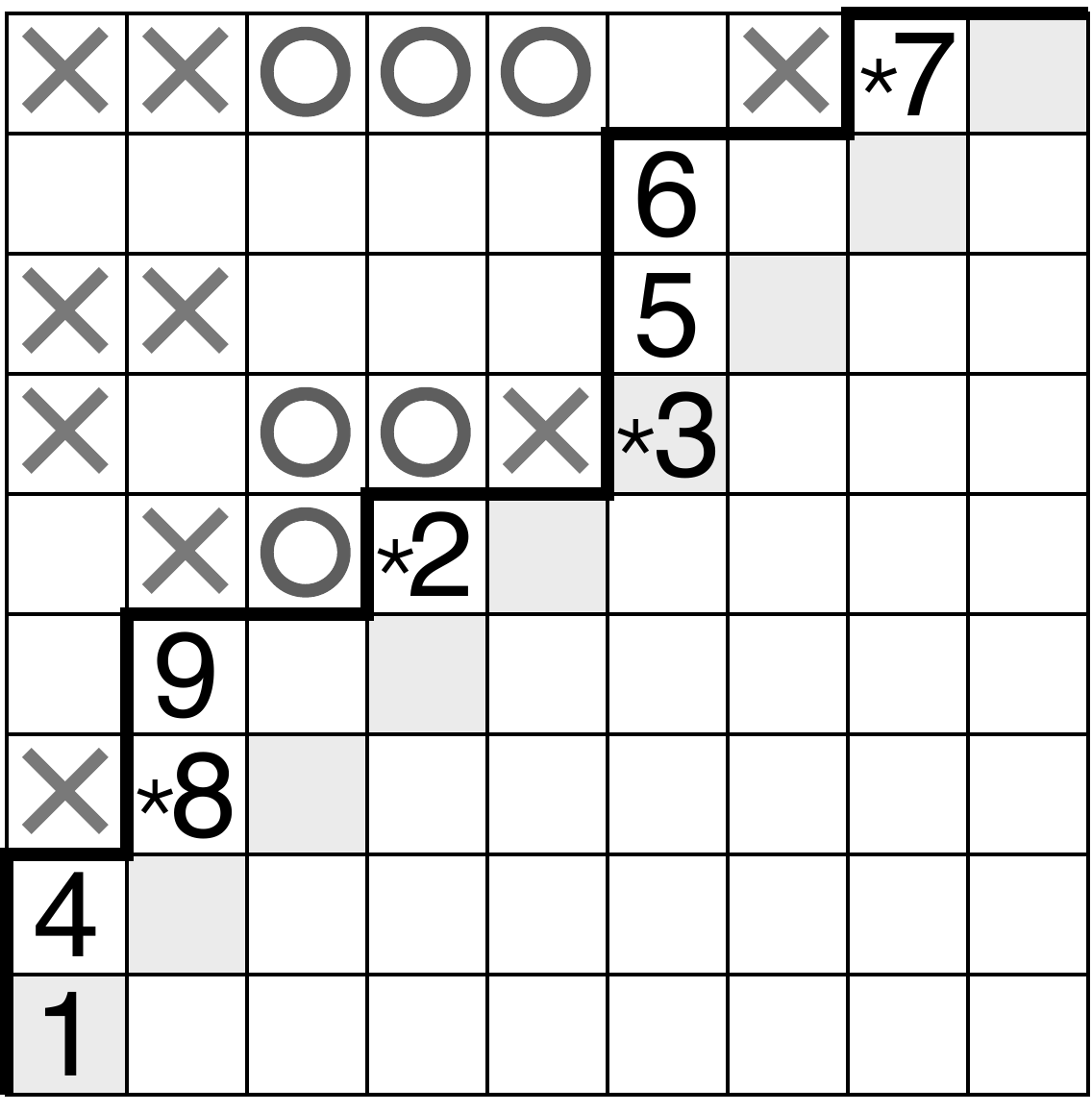}
\caption{A marked parking function of size 9 with the images of $\phi$ and $\psi$ indicated.}
\label{fig:dimineq}
\end{center}
\end{figure}

In Figure \ref{fig:dimineq} we demonstrate this construction by marking the cells in the image of $\phi$ with $X$'s and those in the image of $\psi$ with $O$'s. As an exception, we show valley markings in the same cell as the adjacent car to avoid overlapping symbols. Here the diagonal inversions come from the pairs $(1,3)$, $(4,3)$, $(4,5)$, $(4,7)$, $(4,8)$, $(5,7)$, $(9,2)$, $(9,5)$, $(9,7)$.

Taking the union of the cells contributing to $\area$, those in the image of $\phi$, and those in the image of $\psi$, gives $\area + \dinv + k + {k \choose 2}$ cells above the main diagonal. Since the total number of cells above the main diagonal is ${n \choose 2}$, this proves the desired inequality. 
\end{proof}

\begin{theorem} \label{thm:abc}
For any given $a,b,c$ such that $a+b+{c+1 \choose 2} \leq {n \choose 2}$, there is at least one marked parking function $M\ns P\ns F$ of size $n$ with $\area(M\ns P\ns F)=a$, $\dinv(M\ns P\ns F)=b$, and $c$ markings.
\end{theorem}

\begin{proof}
Since $a \leq {n \choose 2}$,
there is some $0 \leq k < n$ and $0 \leq r < n-k-1$ so that 
\begin{equation}
a = (n-1) + (n-2) + \dots + (n-k) + r.
\end{equation} 
Let $D(a)$ be the Dyck path $N^{k} (NE)^{r} (EN)^{n-k-r} E^k$. That is, $D(a)$ has $k$ full diagonals under it and the remaining $r$ cells under the path are all as far left as possible in the $k+1$-st diagonal. There is a unique parking function $P\ns F(a)$ with underlying Dyck path $D(a)$ and word $\sigma = n \, n-1 \, \dots \, 2 \, 1$. (Every Dyck path can be labeled with this word in exactly one way.)

By Theorem \ref{thm:ineq}, there are at least ${c+1 \choose 2}$ cells above $D(a)$. Hence there are at least $c+1$ cars in the highest two diagonals of $P\ns F(a)$, all but the leftmost of which occur next to valleys. Form a marked parking function $M\ns P\ns F(a,c)$ by marking the rightmost $c$ of these cars in $P\ns F(a)$. See Figure \ref{fig:abc} for an example.

It is not hard to see from the marking algorithm in the proof of Theorem \ref{thm:ineq} that every cell above $D(a)$ is either in the image of $\phi$ or $\psi$, and hence that $\area(M\ns P\ns F(a,c)) + \dinv(M\ns P\ns F(a,c)) + {c+1 \choose 2} = {n \choose 2}$. Let $\Pi(a,c)$ be the schedule of $M\ns P\ns F(a,c)$. By Theorem \ref{thm:MPFsched}, there is at least one marked parking function $M\ns P\ns F(a,b,c) \in \mathcal{M\ns P\ns F}(\Pi(a,c))$ with $\dinv(M\ns P\ns F(a,b,c)) = b \leq \dinv(M\ns P\ns F(a,c))$. 
\end{proof}

\begin{figure}
\begin{center}
\includegraphics[width=1.0in]{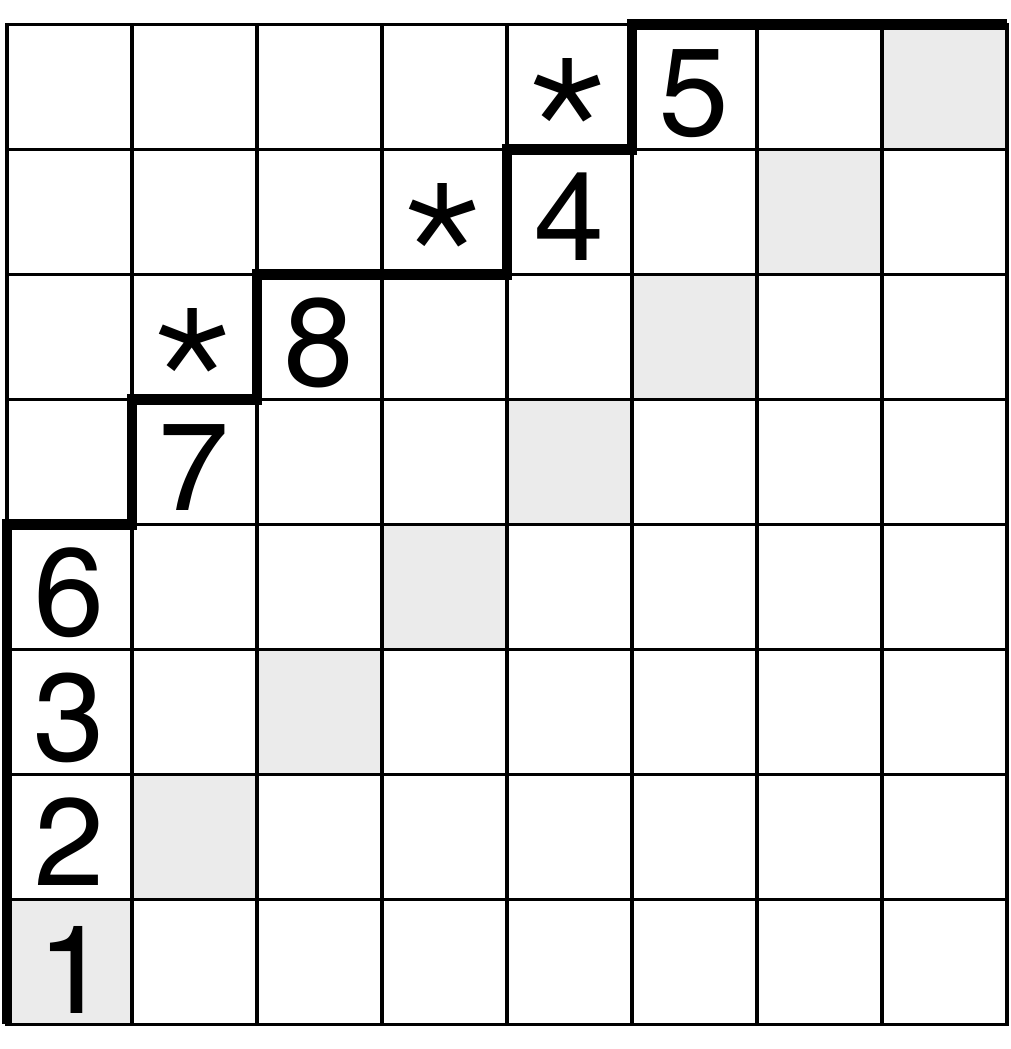}
\caption{The marked parking function $M\ns P\ns F(16,3)$ of size $n=8$}
\label{fig:abc}
\end{center}
\end{figure}

\begin{corollary} \label{cor:degrees}
For $n \geq 1$ and $a,b,c \geq 0$, there is at least one element of $\mathcal{B}_n$ with $x$-degree $a$, $y$-degree $b$ and $\theta$-degree $c$ if and only if $a+b+{c+1 \choose 2} \leq {n \choose 2}$.
\end{corollary}

\subsection{The location of the alternants} \label{subsec:alternants}

Let $G_n$ be the polynomial ring $\mathbb{C}[\theta_1,\theta_2,\dots,\theta_n]$ in the Grassmannian $\theta$ variables modulo the ideal generated by $\theta_1 + \theta_2 + \dots + \theta_n$. $S_n$ acts on $G_n$ by permuting the $\theta$ variables, just as in $\textnormal{S\ns D\ns R}_n$.

\begin{lemma}
For all $n \geq 0$,
\begin{equation} \label{eq:FG}
\textnormal{Frob}(G_n;z) = \sum_{k=0}^{n-1} z^k \, s_{n-k,1^k}.
\end{equation}
\end{lemma}

\begin{proof}
Let $G_n^{(k)}$ be the homogeneous part of $G_n$ with total $\theta$-degree $k$, so that $G_n = \bigoplus _{k=0}^{n-1} G_n^{(k)} $. We would like to show that $\textnormal{Frob}(G_n^{(k)})=s_{n-k,1^k}$ for every $k$. This is obviously true when $k=0$.

Let $R_n$ be the coinvariant algebra of $S_n$ (for ordinary commuting variables $x_i$), i.e.,
$
R_n = \mathbb{C} [x_1,\ldots ,x_n]/I_n(X),
$
where $I_n(X)$ is the ideal generated by symmetric functions in the $x_i$ without constant term. Since the commuting/anti-commuting issue only arises in higher degrees, the homogeneous component of total $x$-degree 1, $R_n^{(1)}$, is isomorphic to $G_n^{(1)}$.
Lusztig (unpublished) and Stanley (see {\cite[Prop.~4.11]{stanley}}) showed that
\begin{align}
\label{eq:Stanley}
\textnormal{Frob}(R_n;q) = \sum_{\lambda \vdash n} s_{\lambda} \sum_{T \in \text{SYT}(\lambda)} q^{\text{maj}(T)}
\end{align}
where the inner sum is over standard Young tableaux of shape $\lambda$. Since there is only one tableau, $T$, with $\text{maj}(T)=1$ (namely that of shape $(n-1,1)$ with label $2$ in the second row), we have $\textnormal{Frob}(G_n^{(1)}) = \textnormal{Frob}(R_n^{(1)}) = s_{n-1,1}$. Hence our claim also holds for $k=1$.

Let $V_n = \{ (z_1,\ldots ,z_n) \in \mathbb{C} ^n: z_1+\ldots z_n=0 \}$. $S_n$ acts on $V_n$ by permuting coordinates. According to Exercise 11 in Chapter 7 of Fulton \cite{fulton},
\begin{enumerate}[(a)]
\item As an $S_n$-module, $V_n$ is isomorphic to the Specht module $S^{(n-1,1)}$.
\item For $1\le k < n$, the $k$th exterior power $\wedge^k \, V_n$ is isomorphic to  the Specht module $S^{(n-k,1^k)}$.
\end{enumerate}
Hence $G_n^{(1)} \cong V_n$. By the definition of wedge product, it follows that $G_n^{(k)} \cong \wedge^k \, G_n^{(1)} \cong \wedge^k \, V_n$ for $1 \leq k < n$. Hence $\textnormal{Frob}(G_n^{(k)}) = s_{n-k,1^k}$ for all $k$ as desired.
\end{proof}

Notice that $\textnormal{S\ns D\ns R}_n$ is a quotient of the tensor product $\textnormal{D\ns R}_n \otimes G_n$ under the diagonal $S_n$-action. Hence any homogeneous component of $\textnormal{D\ns R}_n \otimes G_n$ which does not have an alternant (i.e., does not have nonzero coefficient of $s_{1^n}$) cannot have one in $\textnormal{S\ns D\ns R}_n$ either.

\begin{theorem} \label{thm:bigIneq}
For $a,b,c \geq 0$,
\begin{equation}
\hbox{if } \quad \textnormal{Frob}(\textnormal{D\ns R}_n \otimes G_n) \Big|_{t^a q^b z^c \, s_{1^n}} \neq 0,
\quad \hbox{ then } \quad
a + b + {c+1 \choose 2} \leq {n \choose 2}.
\end{equation}
\end{theorem}

\begin{proof}
The Frobenius character of a tensor product under a diagonal action is simply the Kronecker product $(*)$ of the individual Frobenius characters \cite{bessenrodt}. Hence
\begin{equation}
\textnormal{Frob}(A \otimes B) = \sum_{\nu} s_{\nu} \, \sum_{\lambda,\mu} \ \langle \textnormal{Frob}(A), s_\lambda \rangle \cdot \langle \textnormal{Frob}(B), s_\mu \rangle \cdot \langle s_\lambda * s_\mu, s_\nu \rangle.
\end{equation}
The Kronecker coefficients
$
g(\lambda,\mu,\nu) = \langle s_\lambda * s_\mu, s_\nu \rangle
$
also appear in the identity
\begin{equation}
s_\nu (XY) = \sum_{\lambda,\mu \vdash |\nu|} g(\lambda,\mu,\nu) s_\lambda (X) s_\mu (Y)
\end{equation}
in the expansion of a Schur function in a product $XY$ of sets of variables.  See  \cite{macbook}. Since we are only interested in the case $\nu=(1^n)$, we can make use of the dual Cauchy identity
\begin{equation}
s_{(1^n)} (XY) = \sum_{\lambda} s_{\lambda} (X) s_{\lambda ^{\prime}} (Y)
\end{equation}
to see that
\begin{equation}
g(\lambda,\mu,1^n) = \begin{cases} 1 & \hbox{if }\mu=\lambda' \\
0 & \hbox{otherwise.} \end{cases}.
\end{equation}
This combined with \eqref{eq:FG} gives
\begin{equation}
\langle \textnormal{Frob}( \textnormal{D\ns R}_n \otimes G_n ), s_{1^n} \rangle = \sum_{k=0}^{n-1} \, \langle \nabla e_n, s_{k+1,1^{n-k-1}} \rangle \, z^k.
\end{equation}
Hence we must understand the coefficients of hook Schur functions in $\nabla e_n$. Luckily, there is already a combinatorial interpretation for these coefficients. Below, we briefly describe the combinatorial objects involved and modify our earlier arguments to get the required ${n \choose 2}$ inequality. See Theorems \ref{thm:hagSchroder} and \ref{thm:SchroderIneq}.
\end{proof}

\begin{remark}
Finding a combinatorial interpretation for the Kronecker coefficients is an important unsolved problem, but some special cases are solved. In particular, when either $\lambda$ or $\mu$ is a hook, Blasiak \cite{JKron} gives a combinatorial interpretation for the Kronecker coefficients in terms of colored Yamanouchi tableaux. This result is exactly what is needed to understand all the Kronecker coefficients appearing in our formula for $\textnormal{Frob}(\textnormal{D\ns R}_n \otimes G_n)$. While we don't have a positive combinatorial formula for all the terms $\langle \nabla e_n, s_\lambda \rangle$, one can get a signed interpretation by applying Egge, Loehr, and Warrington's \cite{FtoS} result about Gessel quasisymmetric functions to the Shuffle Theorem. See Garsia and Remmel \cite{GRftos} for a practical guide to this method. It would be interesting to use these tools to better understand the role of Conjecture \ref{conj:ineq}'s inequality in the larger module $\textnormal{D\ns R}_n \otimes G_n$.
\end{remark}

Egge, Killpatrick, Kremer, and the first author \cite{EHKK} express $\langle \nabla e_n, s_{(n-k,1^k)} \rangle$ as a weighted sum of Schr\"oder paths. These are similar to Dyck paths, but in addition to North steps $(0,1)$ and East steps $(1,0)$, one may also make diagonal steps $(1,1)$. The path must still stay weakly above the main diagonal. The $\area$ of a Schr\"oder path is the same as the area of the Dyck path obtained by replacing each diagonal step by $NE$. The $\dinv$ of a Schr\"oder path is the number of pairs of North steps which are ``attacking" in the same sense as for parking functions. That is, two North steps are attacking if they are in the same diagonal or in consecutive diagonals with the leftmost one in the higher diagonal. In Figure \ref{fig:Schroder}, we show a Schr\"oder path with some additional markings. Here $\area = 8$ and $\dinv = 7$.

Following the conventions in \cite{hagbook}, let $\widetilde{L}^{+}_{n,n,d}$ be the set of Schr\"oder paths consisting of $n-d$ North steps, $n-d$ East steps, and $d$ diagonal steps so that there is no diagonal step above the highest North step. For example, the path in Figure \ref{fig:Schroder} is an element of $\widetilde{L}^{+}_{9,9,4}$. Then \cite{EHKK} conjectured and the first author \cite{hagSchroder} proved the following.

\begin{theorem}[\cite{hagSchroder}] \label{thm:hagSchroder}
For all $n,k \geq 0$,
\begin{equation}
\langle \nabla e_n, s_{k+1,1^{n-k-1}} \rangle = \sum_{P \in \widetilde{L}^{+}_{n,n,k}} t^{\area(P)} q^{\dinv(P)}
\end{equation}
\end{theorem}

\begin{figure}
\begin{center}
\includegraphics[width=1.2in]{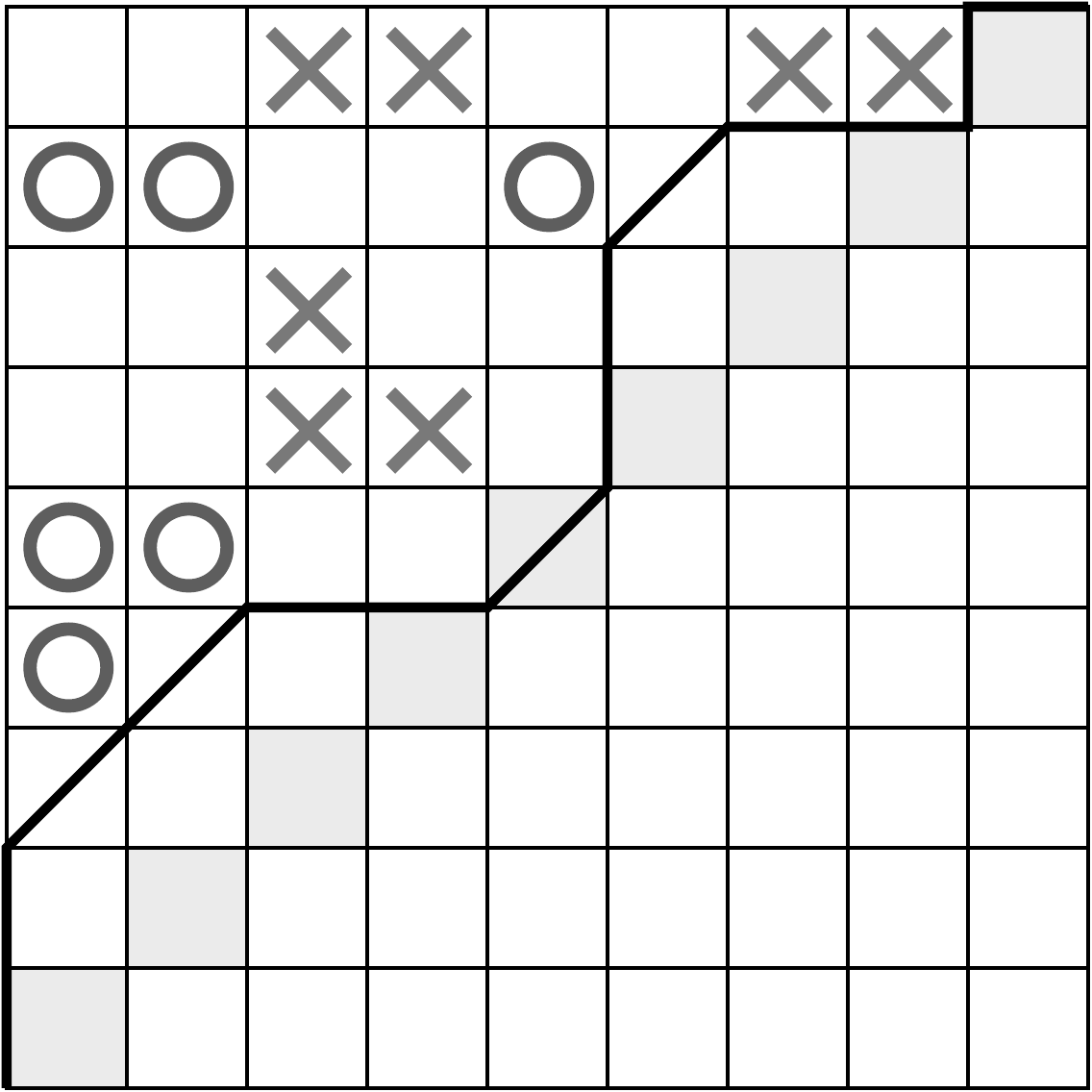}
\caption{A Schr\"oder path of size 9.}
\label{fig:Schroder}
\end{center}
\end{figure}

\begin{theorem} \label{thm:SchroderIneq}
For any $P \in \widetilde{L}^{+}_{n,n,c+1}$, if $\area(P)=a$ and $\dinv(P)=b$, then
\begin{equation}
a + b + {c+1 \choose 2} \leq {n \choose 2}.
\end{equation}
\end{theorem}

\begin{proof}
We use a marking algorithm very similar to the one from the proof of Theorem \ref{thm:ineq}. If $N_1$ and $N_2$ are two attacking North steps such that $N_1$ is further left, follow the left $N_1$'s diagonal to the right until you hit an East step. Directly above this East step and directly left of $N_2$ is a cell which we will denote by $\phi(N_1,N_2)$. If $D_1$ and $D_2$ are two diagonal steps such that $D_1$ is further left, let $\psi(D_1,D_2)$ be the cell directly above $D_1$ and left of $D_2$. In Figure \ref{fig:Schroder} we have marked cells in the image of $\phi$ with an $X$ and cells in the image of $\psi$ with an $O$.

As before, $\phi$ and $\psi$ are injections into the set of cells above $P$ and they have disjoint images. Since the number of diagonal steps is $c+1$, there are ${c+1 \choose 2}$ cells in the image of $\psi$ and $\dinv(P)$-many in the image of $\phi$. Hence the desired inequality holds.
\end{proof}

This completes the proof of Theorem \ref{thm:bigIneq}. From this we obtain support for Conjecture \ref{conj:ineq}:
\begin{corollary} \label{cor:ineq}
For $a,b,c \geq 0$,
\begin{equation}
\hbox{if } \quad \textnormal{Frob}(\textnormal{S\ns D\ns R}_n;q,t,z) \Big|_{t^a q^b z^c \, s_{1^n}} \neq 0,
\quad \hbox{ then } \quad
a + b + {c+1 \choose 2} \leq {n \choose 2}.
\end{equation}
\end{corollary}

\subsection{Extreme $\theta$-degrees} \label{subsec:extreme}

Note that when the $\theta$-degree is $0$, Conjecture \ref{conj:basis} matches Theorem \ref{thm:CObasis} exactly. At the other extreme, if the $\theta$-degree is more than $n$, the pigeonhole principle shows some $\theta_i$ is squared, giving $0$. Additionally, since $\theta_1 + \theta_2 + \dots + \theta_n = 0$, multiplying on the right by $\theta_2 \theta_3 \cdots \theta_n$ shows $\theta_1 \theta_2 \cdots \theta_n = 0$. Hence the maximum $\theta$-degree is $n-1$.

Clearly from the inequality of Theorem \ref{thm:ineq}, if a parking function has $n-1$ markings, it has $\area=\dinv=0$. There is only one such parking function, namely the one with $\area=0$, word $n \, n-1 \, \dots \, 2 \, 1$, and all cars except $1$ marked. Its contribution to the combinatorial side of the Delta Conjecture is $z^{n-1} s_{1^n}$, and indeed on the symmetric function side, $z^{n-1} \Delta^{\prime}_{e_0} e_n = z^{n-1} e_n = z^{n-1} s_{1^n}$.

Now we may use Wallach's proof \cite[Theorem 14]{wallach} of the $b=0$ case of Conjecture \ref{conj:ineq} to see that if the $\theta$-degree of a nonzero monomial is $n-1$, then the $x$- and $y$-degrees are $0$. That is, Wallach shows that a monomial consisting of $n-1$ $\theta$'s and even a single $x$ or $y$ is already equal to $0$. Hence the top component of the super-diagonal coinvariant ring (in terms of $\theta$-degree) consists only of products of $n-1$ distinct $\theta$'s.

Note that any two such products with the same set of indices are equal up to a sign since $\theta_i \theta_j = - \theta_j \theta_i$. Let $S,T \subset [n]$ be any two distinct sets of size $n-1$. There are some numbers $s,t$ so that $S \setminus T = \{s\}$ and $T \setminus S = \{t\}$. Let $U = S \cap T$. Then if we multiply the identity $\theta_1 + \theta_2 + \dots + \theta_n = 0$ on the right by the monomial $\prod_{u \in U} \theta_u$, most terms on the left hand side will have some $\theta_i^2$ and become 0. Hence we obtain
\begin{equation}
\theta_s \cdot \prod_{u \in U} \theta_u \ + \ \theta_t \cdot \prod_{u \in U} \theta_u \ = \ 0.
\end{equation}
Therefore products of $\theta$'s with index set $S$ are equivalent (again up to a sign) to products of $\theta$'s with index set $T$. So this homogeneous component has dimension 1. Our candidate basis $\{ \theta_1 \theta_2 \cdots \theta_{n-1} \}$ is indeed a basis. Moreover, $S_n$ acts on this basis via the sign representation (again since $\theta_i \theta_j = - \theta_j \theta_i$). Hence the graded character is also $z^{n-1} s_{1^n}$ as claimed.

\subsection{Some small cases}

Zabrocki \cite{mikeMod} has checked Conjecture \ref{conj:zabrocki} by computer for $n \leq 6$. So for small cases we know that the dimension of each component matches the number of elements in our candidate basis. Hence it is enough to show that our candidate basis spans the non-empty homogeneous components.

For $n=2$, we have
$$
z \, \Delta^{\prime}_{e_0} e_2 + \Delta^{\prime}_{e_1} e_2 = s_2 + (t+q+z)s_{1^2}.
$$
Below we display the three ordered set partitions of size $n=2$, their marked words $\tau^{*}(\Pi)$, their schedule numbers $\left(w_\Pi(\tau_i)\right)_{i=1}^n$, the sum of the corresponding basis monomials, and their schedule weight.
\begin{equation}
\qquad \quad
\begin{matrix}
\Pi & & \tau^{*}(\Pi) & & (w_\Pi(\tau_i)) & & monomials & & weight \\[2pt]
\hline \\[-12pt]
\{\{1,2\}\} & \ & 1\overset{*}{2} & \ & (1,1) & \ & \theta_2 & \ & z \\[2pt]
\{\{1\},\{2\}\} & & 12 & & (2,1) & & (1+x_1) & & [2]_q \\[2pt]
\{\{2\},\{1\}\} & & 21 & & (1,1) & & y_2 & & t \\[2pt]
\end{matrix}
\end{equation}

Note that the ideal $I_n(X,Y,\Theta)$ contains the relations $x_1+x_2+\cdots+x_n=0$, $y_1+y_2+\cdots+y_n=0$, and $\theta_1+\theta_2+\cdots+\theta_n=0$. Hence in this small case ($n=2$), we see $x_1=-x_2$, $y_1=-y_2$, and $\theta_1=-\theta_2$. So our basis is spanning, as desired, in the case $n=2$.

Now we consider the case $n=3$. We can proceed with hand computations as we did for $n=2$. We have already seen in Section \ref{subsec:extreme} that we get the correct Frobenius character when the $\theta$-degree is $0$ or $\geq n-1 = 2$. Hence we need only address the components with $\theta$-degree $1$. This corresponds to the symmetric function
$$
z \, \Delta^{\prime}_{e_1} e_3 =  z s_{2,1}+zq(s_{2,1}+s_{1,1,1})+zq^2s_{1,1,1} + zt(s_{2,1}+s_{1,1,1})+zqts_{1,1,1}+zt^2 s_{1,1,1}.
$$
As before, we need only check that our candidate basis is spanning in each non-empty component.

The submodule with $\theta$-degree $1$ corresponds to ordered set partitions with $2$ blocks, of which there are $6$.
\begin{equation}
\begin{matrix}
\Pi & & \tau^{*}(\Pi) & & (w_{\Pi}(\tau_i)) & & monomials & & weight \\[2pt]
\hline \\[-12pt]
\{\{1,2\},\{3\}\} & \ & 1 \overset{*}{2} 3 & \ & (2,1,1) & \ & \theta_2(1+x_1) & \ & z[2]_q  \\[2pt]
\{\{1,3\},\{2\}\} & & 3 \overset{*}{1} 2 & & (1,1,1) & & \theta_1y_3 & & zt \\[2pt]
\{\{2,3\},\{1\}\} & & 2 \overset{*}{3} 1 & & (1,1,1) & & \theta_3y_2y_3 & & zt^2 \\[2pt]
\{\{1\},\{2,3\}\} & & 1 2 \overset{*}{3} & & (2,1,2) & & \theta_3(1+x_1)(1+x_3) & & z[2]_q^2 \\[2pt]
\{\{2\},\{1,3\}\} & & 2 1 \overset{*}{3} & & (1,1,1) & & \theta_3y_2 & & zt \\[2pt]
\{\{3\},\{1,2\}\} & & 3 1 \overset{*}{2} & & (1,1,2) & & \theta_2y_3(1+x_2) & & zt [2]_q \\[2pt]
\end{matrix}
\end{equation}

We can write all monomials from each component in terms of the candidate basis elements using the relations
\begin{equation}
\theta_1^a y_1^b x_1^c + \theta_2^a y_2^b x_2^c + \theta_3^a y_3^b x_3^c = 0
\end{equation}
for $a,b\geq0$, $c=0,1$, and $(a,b,c) \neq (0,0,0)$.

The component with $x$-degree 0, $y$-degree 0 and $\theta$-degree 1 contains two elements of $\mathcal{B}_3$, namely $\theta_2$ and $\theta_3$. This is indeed a spanning set since $\theta_1 + \theta_2 + \theta_3 = 0$.

Our candidate basis $\mathcal{B}_3$ has three elements in the component with $x$-degree $0$, $y$-degree $1$, and $\theta$-degree $1$, namely $\theta_1 y_3$, $\theta_3 y_2$, and $\theta_2 y_3$. First notice
\begin{equation}
\theta_3 y_3 = -\theta_1 y_3 - \theta_2 y_3.
\end{equation}
Therefore
\begin{equation}
\theta_3 y_1 = -\theta_3 y_2 -\theta_3 y_3 = \theta_1 y_3 - \theta_3 y_2 + \theta_2 y_3.
\end{equation}
Furthermore
\begin{align}
\theta_1 y_1 &= -\theta_2 y_2 - \theta_3 y_3 \\
 &= \left( \theta_2 y_1 + \theta_2 y_3 \right) + \left( \theta_3 y_1 + \theta_3 y_2 \right)  \\
 &= -\theta_1 y_1 + \theta_2 y_3 + \theta_3 y_2,
\end{align}
which implies
\begin{equation}
\theta_1 y_1 = \frac{1}{2} \theta_2 y_3 + \frac{1}{2} \theta_3y_2.
\end{equation}
Hence
\begin{equation}
\theta_1 y_2 = -\theta_1 y_1 - \theta_1 y_3 = - \theta_1y_3 -\frac{1}{2} \theta_2 y_3 + \frac{1}{2} \theta_3y_2.
\end{equation}
Also
\begin{align}
\theta_2 y_1 &= -\theta_2 y_2 - \theta_2 y_3 \\
&= \theta_1 y_1 + \theta_3 y_3 - \theta_2 y_3 \\
&= \theta_1y_1 - \theta_1 y_3 - 2\theta_2 y_3 \\
&= -\theta_1 y_3 + \frac{1}{2} \theta_3y_2 - \frac{3}{2} \theta_2 y_3.
\end{align}
Finally
\begin{equation}
\theta_2 y_2 = -\theta_2 y_1 - \theta_2 y_3 = \theta_1 y_3 - \frac{1}{2} \theta_3y_2 + \frac{1}{2} \theta_2 y_3.
\end{equation}
So $\{ \theta_1 y_3, \theta_3 y_2, \theta_2 y_3\}$ is indeed a spanning set for this homogeneous component.

We may continue with this type of hand computation for the remaining $n=3$ components. But this becomes infeasible for $n=4$. It is also possible to continue via computer. We used MAPLE to check that our candidate basis is indeed spanning for all $n=3$ and $n=4$ components. Beyond that, the task is already computationally challenging. We checked many (but not all) components for $n=5$, and they were spanned by our candidate basis.


\section*{Acknowledgements}

The authors gratefully acknowledge NSF support for this work; the first author by grant DMS-1600670 and the second by grant DMS-1603681.
The authors would also like to thank Mike Zabrocki for suggesting the method used in Section \ref{subsec:alternants} (private correspondence), as well as Alessandro Iraci and Anna Vanden Wyngaerd for helping us to improve our exposition.

\bibliographystyle{alpha}
\bibliography{Delta}

\end{document}